\documentclass[11pt]{article}
\usepackage{pb-diagram}
\usepackage{amsmath}
\usepackage{amsthm}
\usepackage{mathtools}
\usepackage{amsfonts}
\usepackage{amssymb}
\usepackage{graphicx}
\usepackage[usenames]{color}
\usepackage{xcolor}
\usepackage{hyperref}
\usepackage{soul}
\usepackage{enumitem}
\usepackage{amsmath}
\usepackage{amssymb}
\usepackage{amsfonts}
\usepackage{amscd}
\usepackage{hyperref}

\usepackage{tocloft}
\usepackage{xcolor}

\usepackage{hyperref}
\hypersetup{linktocpage}

\hypersetup{colorlinks,
    linkcolor={red!50!black},
    citecolor={blue!80!black},
    urlcolor={blue!80!black}}
\usepackage{float}

\newtheorem{theorem}{Theorem}[section]
\newtheorem{lemma}[theorem]{Lemma}
\newtheorem{cor}[theorem]{Corollary}

\newtheorem{prop}[theorem]{Proposition}

\theoremstyle{definition}
\newtheorem{definition}[theorem]{Definition} 
\newtheorem{remark}[theorem]{Remark}
\newtheorem{example}[theorem]{Example}

\newcommand{\be}{\begin{enumerate}}
\newcommand{\ee}{\end{enumerate}}
\newcommand{\beq}{\begin{equation}}
\newcommand{\eeq}{\end{equation}}

\def\N{{\mathbb{N}}}
\def\Z{{\mathbb{Z}}}

\def\MA{{\mathbb{A}}}
\def\MB{{\mathbb{B}}}

\def\A{{\mathcal{A}}}

{}
{}
{}

\newcommand{\Ker}{\operatorname{Ker}}
\newcommand{\Aut}{\operatorname{Aut}}

\newcommand{\NN}{\mathbb{N}}
\newcommand{\ZZ}{\mathbb{Z}}

\newcommand{\G}{\mathcal G}
\renewcommand{\wr}{\,{\rm wr}\,}
\newcommand{\e}{\varepsilon}

\newcommand{\ModS}{{\rm Mod}_{\mathcal C}}

\title{On the first-order genus of wreath products and their central extensions}
\author{Olga Kharlampovich\footnote{CUNY, Graduate Center and Hunter College} , Alexei Miasnikov\footnote{Stevens Institute of Technology}, Denis Osin\footnote{Vanderbilt University}
\thanks{The first author was supported by the Dolciani foundation. The third author was supported in part by the NSF Grant DMS-2153805. }}

\date{}
\begin{document}

\maketitle

\begin{abstract}
    We prove that groups of the form $\ZZ^m \wr \ZZ^n$, where $m,n \in \NN$, are regularly bi-interpretable with $\ZZ$ and therefore are first-order rigid: every finitely generated group elementarily equivalent to $\ZZ^m \wr \ZZ^n$ is isomorphic to $\ZZ^m \wr \ZZ^n$. On the other hand, we show that $\ZZ^2\wr \ZZ$ admits $2^{\aleph_0}$ elementarily equivalent, pairwise non-isomorphic central extensions with finite kernel.
\end{abstract}

\section{Introduction}

Classification of groups with respect to elementary equivalence is an area of algebra and mathematical logic that goes back to the pioneering works of L\"owenheim, Skolem, and Tarski in the 1920s. The L\"owenheim–Skolem theorem tells us that for a given infinite group $G$ and any infinite cardinal $\lambda$, there is a group $H$ of cardinality $\lambda$ that is elementarily equivalent to $G$ (symbolically, $G \equiv H$). Hence, the first-order theory $Th(G)$ of $G$ -- that is, the set of all sentences in the language of group theory $\{\cdot, {}^{-1}, 1\}$ that are true in $G$ -- does not uniquely characterize the group $G$ up to isomorphism.  This observation gives rise to the first-order classification problem, originally formulated by Tarski: for a given group $G$, characterize all groups elementarily equivalent to $G$. 

This problem turned out to be difficult, so a more tractable question was posed: how many groups $H$ of infinite cardinality $\lambda$ satisfy $H \equiv G$?
  
  The extremal instance of the Löwenheim–Skolem theorem occurs when there exists, up to isomorphism, a unique group $H$ of cardinality $\lambda$ such that $G \equiv H$. Morley termed such theories $Th(G)$ \emph{$\lambda$-categorical} and proved that if $Th(G)$ is categorical in some uncountable cardinal, then it is categorical in all uncountable cardinals (and he did so for arbitrary algebraic structures, not just groups).

Unfortunately, there are not many groups that are uncountably or countably categorical. Furthermore, if the theory $Th(G)$ is unstable, then it follows from Shelah's classification theory that the number of non-isomorphic models of a given uncountable cardinal $\kappa$ is maximal,
 which is $2^\kappa$.

In this case, the theory $Th(G)$ is "wild" or chaotic in all uncountable cardinalities. Furthermore, usually, the theory $Th(G)$ has infinitely many, and quite often uncountably many countable models. For example, the first-order theory of the Heisenberg group $UT_3(\Z)$ has uncountably many countable models. Indeed, standard arithmetic $\langle \N; +, \cdot, 0,1\rangle$ has $2^{\aleph_0}$ countable models $\widetilde{\N}$ with $\widetilde{\N} \equiv \N$, and consequently $UT_3(\N) \equiv UT_3(\widetilde{\N})$. 

Therefore, when viewed through the lens of cardinalities, the standard model-theoretic approach to models of complete theories is, in general, not very informative. A recent promising development in the first-order classification problem arises from within group theory itself. From an algebraic viewpoint, it is natural to consider all finitely generated (rather than merely countable) groups $H$ such that $G \equiv H$. To this end, for a finitely generated group $G$ we define the \emph{first-order genus} $\mathcal{FO}(G)$ as the set of  isomorphism classes $[H]$ of all finitely generated groups $H$ with $G \equiv H$. 

In \cite{ALM1} Avni, Lubotzky, and Mieri termed a finitely generated group $G$ \emph{first-order rigid} if for any finitely generated group $H$, if $G \equiv H$ then $G\simeq H$. In other words, $G$ is first-order rigid if and only if $|\mathcal{FO}(G)| =1$. It turned out that there are many first-order rigid groups, despite the fact that there are only a few countably categorical ones, and there are no infinite, finitely generated, countably categorical groups at all \cite{BCM}. Curiously, among first-order rigid groups, there are some that can be completely characterized by a single sentence in the language of group theory. Following Nies \cite{Nies1,Nies2}, a finitely generated group $G$ is termed \emph{quasi-finitely axiomatizable (QFA)} if there exists a sentence $\sigma$ in the language of group theory such that, whenever $\sigma$ holds in a finitely generated group $H$, we have $G \simeq H$. 

Note that finite groups $G$ are not only QFA but also finitely axiomatizable; that is, any group elementarily equivalent to $G$ is in fact isomorphic to $G$. Infinite finitely generated abelian groups are first-order rigid but not QFA. The former is straightforward, while the latter follows from Baur–Monk quantifier elimination for abelian groups (see, for example, \cite{Prest}). 

If $G$ is a finitely generated nilpotent group, then $\mathcal{FO}(G)$ is finite. To see this note that due to Pickel's result  from \cite{Pickel1} there are only finitely many up to isomorphism finitely generated nilpotent groups which have the same (up to isomorphism) finite quotients  as the group $G$.  This set of isomorphism classes is called the \emph{profinite genus} of $G$. Now, for a finitely generated group $H$, if $G \equiv H$, then $H$ is nilpotent and the quotients $G/G^kG_i$ and $H/H^kH_i$, where $G_i, H_i$ are $i$-th terms of the lower central series, and $G^k$ and $H^k$ are the subgroups generated by $k$-th powers of elements in  $G$ and $H$,
 are  elementarily equivalent for each $k$ and $i$ (because verbal subgroups in finitely generated nilpotent groups have finite width, hence definable in these groups by the same formulas) and finite, hence isomorphic. 
 
 Furthermore, most finitely generated nilpotent groups $G$ are QFA. Indeed,  Oger and Sabbagh showed in \cite{OgerSabbagh} that a finitely generated nilpotent group $G$ is QFA if and only if the center of $G$ is contained in the isolator of the commutant of $G$ and due to \cite{GMO} a random finitely generated nilpotent group satisfies this property.  A similar line of reasoning works for polycyclic groups $G$: Sabbagh and Wilson showed in \cite{SW} that if $H$ is a finitely generated group and $G \equiv H$ then $H$ is polycyclic; Oger established in \cite{O2} that if $G$ and $H$ are  polycyclic groups  and $G \equiv H$ then $G$ and $H$ have the same finite quotients; and Grunewald, Pickel, and Segal proved in  \cite{GPS} that there are only finitely many up to isomorphism polycyclic groups $H$  which have the same (up to isomorphism) finite quotients  as the group $G$.  Hence, for a polycyclic group $G$, the set $\mathcal{FO}(G)$ is finite. Similarly to nilpotent groups, Lasserre described in \cite{Lasserre} all QFA polycyclic groups. 

However, this line of reasoning fails for finitely generated metabelian groups, since Pickel showed in \cite{Pickel2} that there exist infinitely many pairwise non-isomorphic, finitely presented metabelian groups sharing the same finite quotients. For finitely generated metabelian groups $G$ that are not polycyclic, little is known about the conditions under which they are elementarily equivalent, about their first-order genus, or about related model-theoretic properties.

 Khelif showed in \cite{Khelif} that the metabelian Baumslag–Solitar groups $BS(1,n)$ for $n>1$, as well as the free metabelian group $M_2$ of rank 2, are QFA. In \cite{KM1}, Kharlampovich and Myasnikov proved that all non-abelian free metabelian groups $M_n$ of finite rank $n$ are QFA, and they provided a complete description of all groups (not necessarily finitely generated) that are elementarily equivalent to $M_n$. Subsequently, in \cite{DM2}, Danyarova and Myasnikov described all groups $H$ such that $H \equiv BS(1,n)$. All these results for metabelian non-polycyclic groups $G$ are obtained via a novel method based on bi-interpretations of $G$ with standard arithmetic $\mathbb{N}$, or equivalently, with the ring of integers $\mathbb{Z}$. 
 
 A more general approach to the QFA property in groups, grounded in the richness of their algebraic structure, is introduced in \cite{KMS}. An algebraic structure is called {\em rich} if its first-order theory has the same expressive power as its weak second-order theory. It was shown in \cite{KMS} that rich groups with an arithmetic multiplication table are QFA, and also that groups regularly bi-interpretable with $\mathbb{Z}$ are rich and have an arithmetic multiplication table, hence QFA.
There are many rings and groups that are bi-interpretable with $\mathbb{Z}$. Examples include finitely generated infinite integral domains \cite{AKNS}; finitely generated fields (with a few exceptions) \cite{DP};  Chevalley groups (with some exceptions) over such rings \cite{BG}; Thompson's groups $F$ and $T$ \cite{Lasserre2}; and centerless, non-uniform, irreducible higher-rank arithmetic lattices in characteristic zero \cite{AvM}. We refer to \cite{KMS,DM1} for details and further examples.
On the other hand we mention here that free non-abelian groups $F$ are not QFA; $|\mathcal{FO}(F)| = \kappa_0$; and all finitely generated groups $H$ with $F \equiv H$ have been completely described \cite{KM3,Sela}.

In this paper we prove two main results. The first is the following theorem, which answers Problem 4.3 from \cite{KMS} in the affirmative.

\begin{theorem}  \label{metab}   Let $A$ and $B$ be free abelian groups of finite ranks. Then $B\wr A$  is regularly injectively bi-interpretable  with $\mathbb Z$.
\end{theorem} 
  The proof of  Theorem \ref{metab} is divided into four parts: interpretability of $G$ in $\Z$ in Section \ref{se:3.1}, interpretability of $\Z$ in $G$ in Section \ref{se:3.3}, bi-interpretability of $G$ and $\Z$ with parameters, and finally, regular bi-interpretability of $G$ and $\Z$ in Sections \ref{se:4}, \ref{se:5}.

\begin{cor} Let $A$ and $B$ be free abelian groups of finite ranks. Then $B\wr A$ is QFA. 
    \end{cor}

Theorem \ref{metab} and results from \cite{DM1}, Section 5.8 imply the following.

\begin{cor} Let $A$ and $B$ be free abelian groups of finite rank. Then $B\wr A$ is rich, prime, atomic, homogeneous, admits elimination of imaginaries with parameters, and for any group $H$, one has $G \equiv H$ if and only if $H = \Gamma(\widetilde{\Z})$, where $G = \Gamma(\Z)$ is the  interpretation of $G$ in $\Z$ from Theorem \ref{metab} and $\widetilde{\Z}$ is an arbitrary ring with $\widetilde{\Z} \equiv \Z$.
\end{cor}

Our second result shows (see below) that there exists a finite central extension $G$ of $\ZZ^2 \wr \ZZ$ such that  $|\mathcal{FO}(G)| = 2^{\aleph_0}$.  To state the second theorem, we need some notation and terminology.

Recall that the space of finitely generated marked groups $\G$ can be informally defined as the set of all pairs $(G,A)$, where $G$ is a group and $A$ is an ordered finite generating set of $G$, endowed with the topology induced by local convergence of Cayley graphs. Given a finitely generated group $G$, we denote by $[G]\subseteq \G $ its \emph{isomorphism class}; that is,
$$
[G]=\{ (H,B)\in \G\mid H\cong G \}.
$$

The following definition is inspired by connections between the topological properties of isomorphism classes in $\G$ and model theory (see \cite{Osi21a}).

\begin{definition}
A finitely generated group $G$ is said to be \emph{condensed} if $[G]$ has no isolated points.
\end{definition}

We prove the following theorem.

\begin{theorem}\label{main}
There exists a condensed group $G$ that splits as a central extension of the form
\begin{equation}\label{ext}
1\longrightarrow \ZZ_2 \longrightarrow G \longrightarrow \ZZ^2 \wr \ZZ  \longrightarrow 1.
\end{equation}
\end{theorem}

There are three immediate consequences of this result. The first one concerns the complexity of the isomorphism relation. Recall that an equivalence relation $E$ on a topological space $X$ is called \emph{smooth} if there is a Polish space $P$ and a Borel map $\beta \colon X\to P$ such that for any $x,y\in X$, we have $xEy$ if and only if $\beta(x)=\beta(y)$. 

In \cite{W}, Williams proved that the isomorphism relation on the space of  $3$-step solvable, finitely generated, marked groups is non-smooth. We strengthen this result as follows.

\begin{cor}\label{Cor:non-smooth}
The isomorphism relation on the space of center-by-metabelian, finitely generated, marked groups is not smooth.
\end{cor}

Informally, our result means that finitely generated center-by-metabelian groups cannot be ``explicitly classified" up to isomorphism using invariants from a Polish space. Note that Corollary \ref{Cor:non-smooth} is optimal in a certain sense. Indeed, the isomorphism relation on the space of finitely generated marked metabelian groups is obviously smooth, as the latter space is countable. 

Another corollary answers the second part of Question 7.2 in \cite{Osi21a}. 

\begin{cor}\label{Cor:Non-Geom}
  The property of being condensed is not geometric, i.e., not stable under quasi-isometry of finitely generated groups.  
\end{cor}

The last corollary sharply contrasts the fact that $\ZZ^2\wr \ZZ$ is QFA.

\begin{cor}\label{Cor:EE}
There exist $2^{\aleph_0}$ finitely generated, pairwise non-isomorphic, elementarily equivalent central extensions of the form (\ref{ext}). 
\end{cor}

\section{Preliminaries}

\subsection{Interpretability and bi-interpretability}
\label{se:2.1} 

One can use the model-theoretic notion of interpretability and bi-interpretability to study structures
elementarily equivalent to a given one. In this paper we are going to do this for the restricted wreath product
${\mathbb Z}^n\wr {\mathbb Z}^m$.
We remind here some precise definitions and several known facts that may not be very familiar to algebraists.

Let $\mathbb{B} = \langle B ; L\rangle$ be a structure. A subset $A \subseteq B^n$ is called {\em definable} in $\mathbb{B}$ if there is a formula $\phi(x_1, \ldots,x_n)$ (without parameters) in $L(\mathbb{B})$ such that  $$A = \{(b_1,\ldots,b_n) \in B^n \mid \mathbb{B} \models \phi(b_1, \ldots,b_n)\}.$$ In this case we denote $A$ by $\phi(B^n)$ or $\phi(\MB)$ and  say that  \emph{$\phi$ defines $A$} in $\mathbb{B}$.  Similarly, an operation $f$ on the subset  $A$ is definable in $\mathbb{B}$ if its graph is definable in $\mathbb{B}$.  A constant $c$ is definable if the relation $x=c$ is definable. An $n$-ary predicate $P(x_1,\ldots ,x_n)$ is definable in $\mathbb{B}$ if the set $\{(b_1,\ldots ,b_n)\in\mathbb{B}^n| P(b_1,\ldots ,b_n) \ { \rm is\  true}\}$
is definable in $\mathbb{B}$.

In the same vein  an algebraic structure $\mathbb{A} = \langle A ;f, \ldots, P, \ldots, c, \ldots\rangle$  is definable in $\mathbb{B}$ if there is a definable subset $A^* \subseteq  B^n$ and operations $f^*, \ldots, $ predicates $P^*, \ldots, $ and constants $c^*, \ldots, $ on $A^*$ all definable in $\mathbb{B}$ such that the structure $\mathbb{A}^* = \langle A^*; f^*, \ldots, P^*, \ldots,c^*, \ldots, \rangle$ is isomorphic to $\mathbb{A}$. (Note that constants $c,\ldots $ belong to the language of  $\mathbb{A}$, they are not parameters.) 
For example, the center of any group $G$ (considered as a structure in the standard group theoretic language) is definable as a group in $G$. 

 In the notation above if $\sim$ is a definable  equivalence relation on a definable subset $ A \subseteq B^n$ then we say that the quotient set $ A/\sim$ is {\em interpretable}\index{interpretable} in $\mathbb{B}$.  Furthermore, an operation $f$ or a predicate $P$ on the quotient set $A/\sim$ is interpretable in $\mathbb{B}$ if the full preimage of its graph in $A$ is definable in $\mathbb{B}$. For example, if $N$ is a normal definable subgroup of a group $G$, then the equivalence relation $x \sim y$ on $G$ given by $xN = yN$ is definable in $G$, so the quotient set $G/N$ of all right cosets of $N$ is interpretable in $G$. It is easy to see that the multiplication induced 
from $G$ on $G/N$ is also interpretable in $G$. This shows that the quotient group $G/N$ is interpretable in $G$. 

\begin{definition} \label{de:interpretable} An algebraic  structure $\mathbb{A} = \langle A ;f, \ldots, P, \ldots, c, \ldots\rangle$  is\textit{ absolutely interpretable} (or \textit{$0$-interpretable})  in a structure $\mathbb{B}$  if there is a  subset $A^* \subseteq B^n$  definable in $\mathbb{B}$, an equivalence relation $\sim$ on $A^*$ definable in $\mathbb{B}$, operations  $f^*, \ldots, $ predicates $P^*, \ldots, $ and constants $c^*, \ldots, $ on the quotient set $A^*/{\sim}$ all interpretable in $\mathbb{B}$ such that the structure $\mathbb{A}^* = \langle A^*/{\sim}; f^*, \ldots, P^*, \ldots,c^*, \ldots, \rangle$ is isomorphic to $\mathbb{A}$.
 \end{definition}

Now we introduce some  notation. An interpretation  of $\MA$ in $\MB$ is described  by the following set of formulas in the language $L(\MB)$
$$
\Gamma  = \{U_\Gamma(\bar x), E_\Gamma(\bar x_1, \bar x_2), Q_\Gamma(\bar x_1, \ldots,\bar x_{t_Q}) \mid Q  \in L(\MA)\}
$$
(here $\bar x$ and $\bar x_i$ are $n$-tuples of variables)  which  interpret $\mathbb{A}$ in $\mathbb{B}$ (as in the definition \ref{de:interpretable} above). Namely,  $U_\Gamma$  defines in $\mathbb{B}$ a subset  $A_\Gamma  = U_\Gamma(B^n)  \subseteq B^n$, $E_\Gamma$  defines in $\MB$ an  equivalence relation $\sim_\Gamma$ on $A_\Gamma$, and the formulas $Q_\Gamma$ define   functions $f_\Gamma$, predicates $P_\Gamma$, and constants $c_\Gamma$ that interpret the corresponding symbols from $L(\mathbb{A})$ on the quotient set $A_\Gamma/\sim_\Gamma$ in such a way that the $L$-structure $\Gamma(\MB) = \langle A_\Gamma/\sim_\Gamma; f_\Gamma, \ldots, P_\Gamma, \ldots, c_\Gamma, \ldots \rangle $ is isomorphic to $\MA$.  Note, that   we interpret a constant $c \in L(\MA)$ in the structure $\Gamma(\MB)$ by the $\sim_\Gamma$-equivalence  class of some tuple $\bar b_c \in A_\Gamma$ defined in $\MB$ by the formula $Q_c$. We write $\mathbb{A} \simeq \Gamma(\mathbb{B})$ if $\Gamma$ interprets $\MA$ in $\MB$ as described above and refer to $\Gamma$ as an {\em interpretation code}\index{interpretation code} or just {\em code}.  The number $n$ is called the dimension of $\Gamma$, denoted $n = dim\Gamma$. By $\mu_\Gamma$ we denote a a surjective map $A_\Gamma \to \MA$ (here $\MA = \langle A;L(\MA)\rangle$) that gives rise to an isomorphism   $\bar \mu_\Gamma: \Gamma(\MB) \to \mathbb{A}$.  We refer to  this map $\mu_\Gamma$ as  \emph{the coordinate map}\index{coordinate map} of the interpretation $\Gamma$. Sometimes we call the relation $\sim_\Gamma$  the \emph{kernel}  of the coordinate map $\mu_\Gamma$ and denote it by $\ker(\mu_\Gamma)$. Finally, notation $\mu: \MB \rightsquigarrow \MA$  means that  $\MA$ is interpretable in $\MB$ with the coordinate map $\mu$.
We use this notation throughout the paper.

More generally, the formulas that interpret $\mathbb{A}$ in $\mathbb{B}$ may contain elements from $\mathbb{B}$ that are not in the language $L(\mathbb{B})$, i.e., some parameters, say $p_1, \ldots,p_k \in B$.  In this case we assume that all the formulas from the code $\Gamma$ have a tuple  of extra variables $\bar y = (y_1, \ldots,y_k)$  for parameters in $\MB$: 
\begin{equation} \label {eq:code}
\Gamma =  \{U_\Gamma(\bar x,\bar y), E_\Gamma(\bar x_1, \bar x_2,\bar y), Q_\Gamma(\bar x_1, \ldots,\bar x_{t_Q},\bar y) \mid Q \in L(\mathbb{A})\}
\end{equation}
so that after the assignment $y_1 \to p_1, \ldots,y_k \to p_k$ the code  interprets $\mathbb{A}$ in $\mathbb{B}$.  In this event we write $\mathbb{A} \simeq \Gamma(\mathbb{B},\bar p)$ (here $\bar p = (p_1, \ldots,p_k)$), and say that $\mathbb{A}$ is interpretable in $\mathbb{B}$  by the code $\Gamma$ with parameters $\bar p$. In the case when $\bar p = \emptyset$ one gets again the absolute interpretability.

We will say that a subset $D \subseteq A_\Gamma/\sim_\Gamma$ is definable in $\MB$ if its full preimage in $ A_\Gamma$ is definable in $\MB$. More generally, a subset $D \subseteq (A_\Gamma/\sim_\Gamma)^m$ is definable in $\MB$ if its full preimage in $A_\Gamma^m$ under the natural projection $A_\Gamma^m \to (A_\Gamma/\sim_\Gamma)^m$ is definable in $\MB$.



Now we discuss a very strong version of  mutual interpretability of two structures, so-called {\em bi-interpretability}. 

\begin{definition} \label{de:bi-inter}
Two algebraic structures  $\MA$ and $\MB$ are called {\em bi-interpretable}\index{bi-interpretable} (with parameters) in each other  if  the following conditions hold:
\begin{itemize}
\item [1)] $\MA$ and $\MB$ are interpretable (with parameters) in each other, so $\MA \simeq \Gamma(\MB,p)$ and $\MB \simeq \Delta(\MA,q)$ for some codes $\Gamma$ and $\Delta$ and tuples of parameters $p, q$. By transitivity  $\MA$, as well as  $\MB$,  is    interpretable (with parameters) in itself, so $\MA \simeq (\Gamma \circ \Delta)(\MA,p^\ast)$ and $\MB \simeq (\Delta \circ \Gamma)(\MB,q^\ast)$, where $\circ$ denotes composition of interpretations and $p^\ast$,  $q^\ast$  the corresponding parameters.
\item [2)]  There is a formula $\theta_\MA(\bar u, x, \bar s)$ in the language $L(\MA)$ such that $\theta_\MA(\bar u, x, p^*)$ defines in $\MA$ the  isomorphism  $ \bar \mu_{\Gamma \circ \Delta}: (\Gamma \circ \Delta)(\MA,p^\ast)  \to  \MA$ (more precisely, it defines the coordinate map $\mu_{\Gamma \circ \Delta}: A_{\Gamma \circ \Delta} \to A$). Similarly, there is a formula $\theta_\MB(\bar v, x, \bar t)$ in the language $L(\MB)$ such that $\theta_\MB(\bar v, x,  q^*)$ defines in $\MB$ the isomorphism $\bar \mu_{ \Delta \circ \Gamma}: (\Delta \circ \Gamma)(\MB,q^\ast) $ (more precisely, it defines the coordinate map $\mu_{ \Delta \circ \Gamma}: B_{(\Delta \circ \Gamma)} \to B$).

\end{itemize}
\end{definition}
\begin{definition} \label{de:regular-bi-int}
Two algebraic structures  $\MA$ and $\MB$ are called {\em regularly  bi-interpretable} in each other  if  the following conditions hold:
\begin{itemize}
\item [1)] $\MA$ and $\MB$ are regularly interpretable  in each other, so $\MA \simeq \Gamma(\MB,\phi)$ and $\MB \simeq \Delta(\MA,\psi)$ for some codes $\Gamma$ and $\Delta$ and the corresponding formulas $\phi, \psi$ (without  parameters). By transitivity  $\MA$, as well as  $\MB$,  is    regularly interpretable  in itself, so $\MA \simeq (\Gamma \circ \Delta)(\MA,\phi^\ast)$ and $\MB \simeq (\Delta \circ \Gamma)(\MB,\psi^\ast)$, where $\circ$ denotes composition of interpretations and $\phi^\ast, \psi^\ast$ the corresponding formulas.
\item [2)]  There is a formula $\theta (\bar y, x, \bar z)$ in the language of $\MA$ such that for every tuple $p^\ast$ satisfying $\phi^\ast(\bar z)$ in $\MA$ the formula $\theta (\bar y, x,p^\ast)$ defines in $\MA$ the isomorphism $\bar \mu_{\Gamma \circ \Delta} : (\Gamma \circ \Delta)(\MA,p^\ast) \to \MA  $  and there is  a formula $\sigma (\bar u, x, \bar v)$ in the language of $\MB$ such that for every tuple $q^\ast$ satisfying $\psi^\ast(\bar v)$ in $\MB$ the formula $\sigma (\bar u, x, q^\ast)$ defines in $\MB$ the  isomorphism $\bar \mu_{\Delta\circ \Gamma} :  (\Delta\circ \Gamma)(\MB,q^\ast)  \to \MB$.

\end{itemize}
\end{definition}

Algebraic structures  $\MA$ and $\MB$ are called {\em 0-bi-interpretable} or {\em absolulutely  bi-interpretable}\index{absolulutely  by-interpretable} in each other if in the definition above the tuples of parameters  $p$ and $q$ are empty.  
But 0-bi-interpretability is rather rare. Indeed, 
 \cite{Hodges} if $\MA$ and $\MB$ are 0-bi-interpretable in each other then their groups of automorphisms are isomorphic.

\subsection{Interpretability in $\mathbb N$ and $\mathbb Z$.}
\label{se:2.2} 

By Lagrange’s Four Square Theorem, every positive integer is a sum of four squares; hence $\N$ is a definable subset of $\Z$. In particular, $\N$ is absolutely interpretable in $\Z$. Note that every element of $\N$ (of $\Z$) is 0-definable in $\N$ (in $\Z$); therefore, every algebraic structure interpretable in $\N$ (in $\Z$) with parameters is also absolutely interpretable in $\N$ (in $\Z$).   On the other hand,  it is known (and easy) that $\Z$ is also interpretable in $\N$. This can be done in a number of different ways. For example, non-negative integers $n$ can be interpreted in $\N$ by pairs $(0,n)$, while negative integers $m$ can be interpreted by pairs  $(1,|m|)$. It is not hard to write formulas for such pairs that would interpret addition and multiplication in $\Z$, so $\N$ and $\Z$ are mutually interpreted in each other. In fact, the following is true (and easy to prove; see, for example, \cite[Theorem 1.5.13]{Kaye})

\begin{prop}
 $\N$ and $\Z$ are absolutely bi-interpretable in each other. 
\end{prop}

From the transitivity of interpretations (see \cite{DM1}), it follows that an algebraic structure  $\A$ is interpretable (regularly interpretable) in $\N$ if and only if it is interpretable (regularly interpretable) in $\Z$. 
It is convenient because many classical results are about interpretations in $\N$, but in groups, interpretations with $\Z$ are more natural. 

From the arithmetical hierarchy and its relation to Turing degrees (usually referred to as Post's theorem), it follows that  every computably enumerable subset of $\N^n$,  $n \geq 1$, is absolutely definable in $\N$ (see \cite{Rogers,Soare}). This provides a very useful tool for interpretability of algebraic structures in $\N$. To make it more precise, recall  that an algebraic structure $\A = \langle A;L\rangle$ in a finite language $L$ is \emph{computable} or \emph{recursive} (see \cite{Rabin,Malcev}) if there is a surjective map $\nu: \N \to \A$, called an enumeration of $\A$,  such that the full preimages with respect to $\nu$ of all the graphs of the basic operations and predicates, including the equality =, on $\A$ are computable in $\N$. 
For  such an enumeration $\nu$ there is a computable function $\lambda:\N \to \N$, where $\lambda(i) = j$ if $j$ is the smallest natural number such that $\nu(j) = \nu(i)$. Put $\N_0 = \lambda(\N)$. Then the restriction $\nu_0$ of $\nu$ to $\N_0$ gives a bijective enumeration  $\nu_0: \N_0 \to A$ such that the full $\nu_0$-preimages of the basic operations and predicates on $A$ are computable in $\N$. Since $\nu_0$ is a bijection, this construction gives an $L$-structure denoted by $\A_\nu$ on $\N_0$ that is isomorphic to $\A$ with respect to $\nu_0$. Since the basic operations and predicates are computable in $\N$, they are definable in $\N$; hence, $\A_\nu$ is an interpretation of $\A$ in $\N$ with the coordinate map $\nu_0$. 

In particular,  a finitely generated  group $G$ with a decidable word problem is computable, since one can effectively enumerate all group words in a fixed finite set of generators of $G$, thus obtaining an enumeration $\nu:\N \to G$ that makes the group $G$ computable. Therefore, such groups are interpretable in $\N$, hence in $\Z$.

Similarly, every  finitely generated commutative ring with a decidable equality problem is computable; hence, it is interpretable in $\N$ and  in $\Z$.  It follows, for example,  that a ring of Laurent polynomials $\Z[a_1, a_1^{-1}, \ldots,a_n,a_n^{-1}]$ is interpretable in $\N$ and in $\Z$. Furthermore, it was shown in \cite{MN} that $\Z[a_1, a_1^{-1}, \ldots,a_n,a_n^{-1}]$ is bi-interpretable with $\N$.

\subsection{Wreath products of free abelian groups }
\label{se:2.3}

 Let $A$ and  $B$ be free abelian groups with bases $\{a_1,\ldots ,a_m\}$ and $\{b_1, \ldots,b_n\}$, respectively. If not stated otherwise, we view the groups $A$ and $B$ in multiplicative notation.
Denote by $G=B\wr A$ the restricted wreath product of $A$ and $B$. We fix this notation throughout the entire paper.

By construction, the group $G$ is a semidirect product $N \rtimes A$, where $N$ is the direct product  $N \simeq \Pi_{a\in A} B_a$ of isomorphic copies $B_a$ of $B$, termed the base group of this wreath product, and $A$, termed the active group, acts on $N$ by shifting the copies $B_a$ in $N$. Namely if $f \in \Pi_{a\in A} B_a$ is viewed as a function $f: A \to \cup_{a\in  A} B_a$,  then  the action of $a \in A$ on $f$ results  in a function $f^a(x) = f(ax)$ for any $x \in A$. As usual, elements $g \in G$ can be written uniquely  in a normal form $g = af$, where $a \in A, f \in N$, and the multiplication is given by
\begin{equation} \label{eq:mult-G}
    a_1f_1 \cdot a_2f_2 = a_1a_2f_1^{a_2}f_2.
\end{equation}

Note that in this case $f^a = a^{-1}fa$ - the conjugation of $f$ by $a$.
 The groups $A$ and $B$ canonically embed into $G$ via the maps $a \in A \to a\cdot 1_N \in G$ and $b \in B \to f_b$, where $f_b(x) = 1$ for $x \neq 1$ and $f_b(1) = b$. We often identify $A$ and $B$ with their images in $G$, so $B$ is identified with the copy $B_1$.

We list below some known properties of wreath products $G$.

\begin{enumerate}
    \item [1)] The subgroup $N$ of $G$ is precisely the normal closure $N=ncl (B)$  of $B$ in $G$.
    \item [2)] The action of $A$ on $N$ extends linearly to the action of the integer group ring $\Z A$ on $N$, so for $u = \Sigma n_ia_i \in \Z A$ and $t \in N$ $t^u = \Pi (t^{n_i})^{a_i}$. This makes $N$ into $\Z A$-module. Observe, that we use here the multiplicative notation for the action of $\Z A$ on $N$, which matches naturally with the conjugation notation in $G$.
    \item [3)] The ring $\Z A$ is isomorphic to the ring of Laurent polynomials 
    $$
    \Z A \simeq \Z[a_1, a_1^{-1}, \ldots, a_m,a_m^{-1}],
    $$
    so sometimes we identify $\Z A$ with $Z[a_1, a_1^{-1}, \ldots, a_m,a_m^{-1}]$.
    \item [4)] $Z A$-module $N$ is a free $\Z A$-module with basis $b_1, \ldots, b_n$.
    \item [5)] \label{normal-form} Every element $g \in G$ can be written uniquely in the canonical (normal)  form 
    $$
    g = {a_1}^{\gamma _1}\ldots {a_m}^{\gamma_m}{b_1}^{P_1}\ldots {b_n}^{P_n} = \Pi_ia_i^{\gamma_i} \Pi_j b_j^{P_j},$$
where $\gamma_i \in \Z$ and $P_1,\ldots ,P_n$ are Laurent polynomials from $\Z A$. The multiplication of normal forms in $G$ is given by the following formula:
\begin{equation} \label{eq:mult-norm-forms}
    \Pi_ia_i^{\gamma_i} \Pi_j b_j^{P_j}  \cdot \Pi_ia_i^{\beta_i} \Pi_j b_j^{Q_j}  = \Pi_ia_i^{\gamma_i +\beta_i} \Pi_j b_j^{P_j\Pi_i a_i^{\beta_i} +Q_j}
\end{equation}
\end{enumerate}

  The tuple  $p= (a_1,\ldots,a_m,b_1, \ldots,b_n)$ plays a role of a "basis" of $G$ and for $g \in G$  the tuple 
$$
s_p(g) = (\gamma_1, \ldots,\gamma_m, P_1, \ldots,P_n) 
$$
is the tuple of "coordinates"  of $g$, provided $g = \Pi_ia_i^{\gamma_i} \Pi_j b_j^{P_j}$.

For convenience, in this and similar cases, we write sometimes $p = \bar a\cdot \bar b$, where $\bar a = (a_1, \ldots,a_m)$, $\bar b = (b_1, \ldots,b_n)$ and $\cdot $ is concatenation of tuples. Respectively, we write $s_p(g) = \bar \gamma\cdot\bar P$ for the tuple of coordinates of $g$ relative  to $p$. Using this notation  we sometimes write $g = {\bar a}^{\bar \gamma}\cdot {\bar b}^{\bar P}$.

Now we describe other bases of $G$ more formally.

\begin{definition} \label{de:bases}
    A tuple $q = (c_1, \ldots,c_m,u_1,\ldots,u_n) $, where $c_i \in G \smallsetminus N, u_j\in N,$ is termed a basis of $G = B \wr A$ if it satisfies the following equivalent conditions in $G$:
    \begin{enumerate}
        \item [1)]  We fix the tuple $p = (a_1,\ldots,a_m,b_1, \ldots,b_n)$ in $G$. The tuples $p$ and $q$ are automorphically equivalent in $G$, i.e., there is an automorphism of $G$ that maps $p$ to $q$ coordinate-wise.
        \item [2)] $C = \langle c_1, \ldots,c_m\rangle \simeq A$,  $G = C\cdot N$, $C \cap N = 1$, and $N$ is a free $\Z[c_1^{\pm 1}, \ldots,c_m^{\pm 1}]$-module with basis $u_1, \ldots,u_n$, when each $c_i$ acts on $N$ by conjugation.
        \item [3)] Every $g \in G$ can be uniquely written in the form 
        $g = \Pi_i c_i^{\gamma_i} \Pi_j u_j^{P_j}$, where $\gamma_i \in \Z$ and $p_j \in \Z[c_1^{\pm 1}, \ldots, c_m^{\pm 1}]$ and multiplication of these normal forms is given by the formula (\ref{eq:mult-norm-forms}), where each $a_i$ is replaced by $c_i$, and each $b_i$ is replaced by $u_i$.
    \end{enumerate}
\end{definition}
To verify that conditions 1), 2), and 3) above are equivalent, suppose first that an automorphism $\phi$ of $G$ maps $p$ to $q$. Then $\phi(A) = C$, so $C \simeq A$. By Lemma \ref{le:2}, item 2), $N$ is the centralizer of the commutant $G'$ of $G$, hence $\phi(N) = N$.  Clearly,  $G = C\cdot N$ and $C \cap N = 1$. The action of $\Z[a_1, a_1^{-1}, \ldots, a_m,a_m^{-1}]$ on $N$ comes from the conjugation of the elements $a_i$ on $N$, so $\phi$ maps this action to the action of $\Z[c_1^{\pm 1}, \ldots,c_m^{\pm 1}]$ on $N$, and $N$ becomes a free $\Z[c_1^{\pm 1}, \ldots,c_m^{\pm 1}]$-module with basis $u_1, \ldots,u_n$, when each $c_i$ acts on $N$ by conjugation. 
 Hence, 1) implies 2). Assume now that 2) holds; then $U = \langle u_1, \ldots,u_n\rangle$ is isomorphic to $B$. Moreover,   $G \simeq N \rtimes C$ and $G \simeq C \wr U$. Hence, the property 5) above holds for $G \simeq N \rtimes C$ and the formula (\ref{eq:mult-norm-forms}) follows. This proves 3). To see that 3) implies  1) observe that if 3) holds then the coordinate-wise map $\psi_0: q \to p$ extends to a bijection on the normal forms 
$$
\psi: \Pi_i c_i^{\gamma_i} \Pi_j u_j^{P_j} \to \Pi_i a_i^{\gamma_i} \Pi_j b_j^{P_j},
$$
which is a group homomorphism since the multiplication is given by the corresponding formulas for $q$ and for $p$ described in (\ref{eq:mult-norm-forms}). Hence $\psi:G \to G$ is an automorphism, so 1) holds.

 If $q = (c_1, \ldots,c_m,u_1,\ldots,u_m)$ is a basis of $G$ then we call the tuple $\bar c  =(c_1, \ldots,c_m)$ a  \emph{top basis}  of $G$ and $\bar u =(u_1,\ldots,u_m)$ a \emph{bottom basis}.


Now we describe the members $G_n$ of the lower central series of $G$. Recall that $G_1 = G$ and $G_{i+1} = [G,G_i]$ for $1 < i \in \N$.

Denote by $\Delta$  the augmentation  ideal of $\Z A$, that is the kernel of the homomorphism $\Z A \to \Z$ induced by the trivial homomorphism $A \to 1$. As an ideal $\Delta$ is generated by elements $a_1 - 1, \ldots, a_m -1$ of $\Z A$. To see this note  that $a_i^{-1}-1 = \frac{1-a_i}{a_i}$ and $uv-1 = u(v-1)+(u-1)$ for $u,v \in A$. Put $\Delta^0 = 1$ and let  $\Delta^i$ be the $i$th-power of $\Delta$ for $i >0$, i.e., it is an ideal in $\Z A$ generated by all products of the type $(y_1-1) \ldots (y_i-1)$, where $y_j \in \{a_1, \ldots, a_m\}$. Denote by $N^{\Delta^{i}}$ the $\Z A$-submodule of $N$ generated by all elements $u^P$, where $u \in N$ and $P \in \Delta^i$. Put $R = \Z A$, $R_i = \Z A/\Delta^i$, and $N_i = N/ N^{\Delta^i}$. In particular, $R_0 \simeq  R$, $R_1  \simeq \Z$. 

\begin{lemma} \label{le:1}
The following hold in $G$:
\begin{itemize}
    \item [1)] $G_i$ is generated as a normal subgroup in $G$ by all simple left-normed commutators of the type
    $[b_k,a_{j_1},\ldots ,a_{j_{i-1}}]$, where $1\leq k \leq n$, $1\leq j_1, \ldots, j_{i-1} \leq m$;
    
    \item [2)]   $G_i = N^{\Delta^{i-1}}$ for every natural $i > 1$, in particular, $G_2 = [G,G] = N^\Delta$;
   \item [3)] $N_i = N/ N^{\Delta^i} \simeq B \otimes_\Z R_i $  for every natural $i \geq 1$;
    \item [4)] $G/G_i \simeq N_{i-1} \rtimes  A \simeq (B \otimes_\Z R_{i-1}) \rtimes  A$ for every natural $i \geq 1$;
    \item [5)] $G_i/G_{i+1}$ is free abelian of finite rank for every natural $i\geq 1$;
    \item [6)] $G/G_{i+1}$ is torsion-free for every natural $i > 1$.
\end{itemize}

\end{lemma}
\begin{proof}
To prove 1) recall that $G_i$ is generated as a normal subgroup by all simple left-normed commutators of the form $[x_1, \ldots,x_i]$ where $x_k \in \{a_1, \ldots,a_m,b_1, \ldots,b_n\}$, $k = 1,\ldots,i$ (this is true for any group $G$ and $x_i$ from any fixed generating set of $G$). Since $N$ is abelian and $G' \leq N$ any commutator $[x_1, \ldots,x_i]$  where $x_s \in N$ for some $s>2$ is equal to 1. Therefore, $G_i$ is generated as a normal subgroup in $G$ by the commutators  $[b_k,a_{j_1},\ldots ,a_{j_{i-1}}]$, as claimed. 

To prove 2), note that 
$$
[b_k,a_{j_1},\ldots ,a_{j_{i-1}}] = b_k^{(a_{j_1}-1)\ldots (a_{j_{i-1}}-1)},
$$
which implies 2). 

To see 3) observe that $N \simeq B \otimes_Z \Z A$ and $B \otimes_\Z R_i $ is a free $R_i$-module with basis $b_1, \ldots,b_n$. The natural epimorphism $\Z A \to R_i = \Z A/\Delta^i$ turns  $B \otimes_\Z R_i $ into a $\Z A$-module. Now, the identical homomorphism $B \to B$ and the epimorphism $\Z A \to R_i$ give rise to an epimorphism 
$$
\theta: B \otimes_Z \Z A \to B \otimes_\Z R_i.
$$
To prove 3) it suffices to note  that $\ker \theta = N^{\Delta^i}$. 

 Clearly, 2) and 3) imply 4). 
 
 For 5) note that $G_i/G_{i+1}$ is freely generated as an abelian group by elements $[b_k,a_{j_1},\ldots ,a_{j_{i-1}}],$ where $j_1\leq\ldots \leq j_{i-1}, k=1,\ldots ,n.$ Indeed, it follows from 1) and 2) that $G_i/G_{i+1}$ is isomorphic as an abelian group to $N^{\Delta^{i-1}}/N^{\Delta^{i}}$. It also follows from 2) that the images of the elements $b_k^{(a_{j_1}-1)\ldots (a_{j_{i-1}}-1)}$ in the quotient $N^{\Delta^{i-1}}/N^{\Delta^{i}}$ generate this quotient. Since the ring $\Z A$ is commutative we can always reorder factors in ${(a_{j_1}-1)\ldots (a_{j_{i-1}}-1)}$ and assume that $j_1\leq\ldots \leq j_{i-1}$. This set still generates the quotient as an abelian group. Moreover, there are no non-trivial relations between these elements. To see this, it suffices to notice that every polynomial $P \in \Z[a_1, \ldots,a_m]$ can be uniquely decomposed as a polynomial with integer coefficients in  "new variables" $y_i = a_i -1$. Indeed, the map $a_1 \to a_1-1, \ldots, a_m \to a_m -1$ extends to an endomorphism  $\psi$ of  $\Z[a_1, \ldots,a_m]$ (since $\Z[a_1, \ldots,a_m]$ is a free commutative ring with basis $a_1, \ldots,a_m$). In fact, $\psi$  is a bijection because there is an inverse map induced by the map $a_1 \to a_1+1, \ldots, a_m \to a_m +1$. It follows that $P = \psi(\psi^{-1}(P))$ is the required decomposition. It is unique since $\psi$ is an automorphism. Now it is clear that the elements $b_k^{(a_{j_1}-1)\ldots (a_{j_{i-1}}-1)},$ where $j_1\leq\ldots \leq j_{i-1}, k=1,\ldots ,n$  form  a basis of $N^{\Delta^{i-1}}$ modulo $N^{\Delta^{i}}$. This proves 5).

 6) follows from 5) since extensions of torsion-free groups are again torsion-free. 

\end{proof}

\begin{lemma} \label{le:2}
The following holds in $G$:

\begin{itemize}
    \item [1)] $N$ is the centralizer $C_G(g)$ of any non-trivial $g \in N$;
    \item [2)] $N$ is the centralizer $C_G(G')$ of the commutant $G' = [G,G]$;
    \item [3)] $A$ is the centralizer $C_G(a)$ of any non-trivial $a \in A$.
   
\end{itemize}

 \end{lemma}
\begin{proof} For any non-trivial $g\in N$, every element in $N$ commutes with $g$ and every element in $A$ does not commute with $g$. Since every element in $G$ is a product of an element in $N$ and an element in $A$, this gives 1) and 2). Since $A$ is abelian, for any $a\in A$,
$A\subseteq C_G(a)$. And $a$ does not commute with any non-trivial $g\in N$, this gives 3). 
    
\end{proof}
 
\section{Interpretability of $G$ and $\Z$ in each other}
\label{se:3}

\subsection{ Interpretability of $G$ in $\Z$}
\label{se:3.1}
The group $G$ is a finitely generated metabelian group; hence, it has a decidable word problem (see, for example, \cite{Baumslag}). Therefore, as we mentioned in Section \ref{se:2.2}, $G$ is absolutely interpretable in $\Z$. However, we describe a particular interpretation here, as we will use it in the sequel. 

We show the interpretability of $G$ in $\mathbb{Z}$ in two steps. First, we prove that the ring
$$
R = \mathbb{Z}[a_1, a_1^{-1}, \ldots, a_m, a_m^{-1}]
$$
is interpretable in $\mathbb{Z}$. Then, using a fixed basis of $G$, we interpret $G$ in $\mathbb{Z}$ via the coordinate representations of its elements with respect to this basis.

To see that $R$ is absolutely interpretable in $\Z$, note that $R$ is computable (since the word problem in $R$ is decidable); hence, there exists a computable  injective function  $\nu: R \to \N $ such that the set $\nu(R)$ is computable in $\N$ and the images under the map $\nu$ of the ring operations of the ring $R$ are computable in $\N$, i.e., the following operations on $\nu(R)$ (here $Q_i \in R$, $i = 1,2,3$) are computable in $\N$:
$$
k_1 \oplus k_2 = k_3 \Longleftrightarrow \wedge_{i = 1}^3 (k_i = \nu(Q_i)) \wedge (Q_1+Q_2 = Q_3),
$$
$$
k_1 \odot k_2 = k_3 \Longleftrightarrow \wedge_{i = 1}^3 (k_i = \nu(Q_i)) \wedge (Q_1\cdot Q_2 = Q_3).
$$
 Again, this is a general fact about computable algebraic structures. Nevertheless, it is convenient to give a sketch of a particular such enumeration $\nu:R \to \N$. Every polynomial $P \in \Z[a_1, \ldots,a_n]$ can be uniquely presented as an integer linear combination of pair-wise distinct monomials on commuting variables $a_1, \ldots,a_n$:
$$
P = \Sigma_{i=1}^d \gamma_i a_1^{\alpha_{i1} }\ldots a_n^{\alpha_{in}},
$$
where $0\neq \gamma_i \in \Z$ and $d, \alpha_i \in \N$. 
Hence, the polynomial $P$ is uniquely presented by a tuple
\begin{equation} \label{eq:nu(P)}
   u  = (\gamma_1,\alpha_{11}, \ldots\alpha_{1n}, \ldots , \gamma_d,\alpha_{d1}, \ldots, \alpha_{dn}).
\end{equation}
If $Q \in R$, then $Q$ can be uniquely presented in the form $Q = \frac{P}{{\bar a}^{\bar \beta}}$ for some $P \in \Z[a_1, \ldots,a_n]$ and some monomial ${\bar a}^{\bar \beta} = a_1^{\beta_1}\ldots a_n^{\beta_n},  \beta_i \in \N$, such that $gcd(P,{\bar a}^{\bar \beta}) = 1$. It follows that $Q$ can be uniquely presented by a pair of tuples $(u_Q,v_Q)$, where $u_Q = u, v_Q = \bar \beta$. Fix an arbitrary computable bijection 
$$
\tau:\bigcup_{i \in \N} \Z^i \to \N
$$
which enumerates all finite tuples of integers. Then $Q$ is uniquely presented by the pair $(\tau(u_Q),\tau(v_Q)) \in \N^2$ and the set of all such pairs is a computable subset of $\N^2$. For a fixed computable bijection $\tau_2:\N^2 \to \N$, put
$$\nu(Q) = \tau_2((\tau(u_Q),\tau(v_Q))).$$
By construction, the subset $\nu(R)$ is computable in $\N$, and given a number $k \in \nu(R)$, one can algorithmically find the corresponding  Laurent polynomial $Q$ such that $k = \nu(Q)$. Then it is easy to see that the operations $\oplus$ and $\odot$ on $\nu(R)$ are computable in $\N$. To finish the proof, recall from Section \ref{se:2.2} that all the computable operations or predicates on $\N$ are definable in $\N$, and $\N$ is definable in $\Z$. This shows that $R$ is absolutely interpretable in $\Z$.

Now, for the second step, we use notation from Section \ref{se:2.3}. Let  $p = \bar a \cdot \bar b$ be a fixed basis of $G$. Then every element $g \in G$  can be uniquely represented as 
$$
    g = {a_1}^{\gamma _1}\ldots {a_m}^{\gamma_m}{b_1}^{P_1}\ldots {b_n}^{P_n},
    $$
so it  can be uniquely represented by the tuple of its coordinates  with respect to $p$
  \begin{equation}\label{form1}
s_p(g) = (\gamma _1,\ldots \gamma_m,P_1, \ldots ,P_n).
\end{equation}
Hence,  by a tuple of integers 
 \begin{equation}\label{form1B}
s_{p,\nu}(g) = (\gamma _1,\ldots \gamma_m,\nu(P_1), \ldots ,\nu(P_n)).
\end{equation}
As mentioned above, the set $\nu(R)$ is definable in $\Z$; therefore, the set 
$$
S_{p,\nu}= \{s_{p,\nu}(g) \mid g \in G\} \subset \Z^{m+n}
$$ 
is also definable in $\Z$. 
Multiplication of normal forms in $G$ with respect to the basis $p$ is given by the formula (\ref{eq:mult-norm-forms}) from Section \ref{se:2.3}, which shows that the tuple of coordinates $s_p(gh)$  of the product of two elements $g, h \in G$ can be computed from the coordinates $s_p(g)$ and $s_p(h)$. The same holds for tuples $s_{p,\nu}(g)$, $s_{p,\nu}(h)$, and $s_{p,\nu}(gh)$. Hence, there is a formula $\Phi(\bar x_1, \bar x_2, \bar x_3)$ of arithmetic such that 
for any $\bar s_1, \bar s_2, \bar s_3 \in S_{p,\nu}$, the following holds: 
$$
\Z \models \Phi(\bar s_1, \bar s_2, \bar s_3) \Longleftrightarrow s_1 = s_{p,\nu}(g) \wedge s_2 = s_{p,\nu}(h) \wedge s_3 = s_{p,\nu}(gh)
$$
for some $g,h \in G$. Note that in any group, the inversion operation and the identity element are absolutely definable from the multiplication. Hence, the construction above gives an interpretation of $G$ in $\mathbb{Z}$, which we denote by $\Gamma$, so $G \simeq \Gamma(\mathbb{Z})$, with the coordinate map $\mu: s_{p,\nu}(g) \mapsto g$.

 \subsection{Interpretability of $\Z$ in $G$}
\label{se:3.3}

In this section, we show how to  interpret $\Z$ in $G$.  Throughout, we adopt the notation introduced in Section \ref{se:2.3}.

 \begin{lemma} \label{le:3}
    \begin{itemize}
        \item [1)] The subgroup $N$ is 0-definable in $G$ by the formula 
$$\phi_N (x)=\forall y \forall z  [x,[y,z]]=1.$$
\item [2)] The subgroup $A$ is definable in $G$ with parameters by the formula
$$
\phi_A(x) = [x,a] =1,
$$
where $1 \neq a$ is an arbitrary non-trivial element of $A$.
    \end{itemize}
\end{lemma}
\begin{proof}
    It follows immediately from Lemma \ref{le:2}.
\end{proof}

\begin{lemma} \label{le:4} 
Let $1 \neq a \in A$ and $g \in G$. Then the following conditions are equivalent:
\begin{enumerate}
    \item [1)] $g$ belongs to the cyclic subgroup $\langle a\rangle$,

    \item [2)] $[g,a] = 1$ and there is $z \in N$ such that $[b_1,g] =  [z,a]$,
    \item [3)] $[g,a] = 1$ and for any $u \in N$ there exists $z \in N$ such that  $[u,g] =  [z,a]$.
\end{enumerate}
\end{lemma}
\begin{proof}
1) $\Longrightarrow$ 3). Let $g \in \langle a\rangle$, so $g = a^\gamma$ for some $\gamma \in \Z$. 
Clearly $[g,a] =1$. Fix $u \in N$. If $\gamma >0$, then  take 
$$
z = u^{a^{\gamma -1} + a^{\gamma -2} + \ldots + a + 1}.
$$
In this case 
$$
[u,g] = u^{-1}u^{a^\gamma} = u^{a^{\gamma} -1} = 
u^{({a^{\gamma -1} + a^{\gamma -2} + \ldots  +  1})(a-1)} = [z,a],
$$
 as claimed.  If $\gamma <0$ then take 
 $$
z = u^{\frac{ -(a^{|\gamma| -1} + a^{|\gamma| -2} + \ldots +  1)}{a^{|\gamma|}}}.
$$
The case $\gamma = 0$ is obvious. 

Clearly, 3) $\Longrightarrow$ 2). To prove the lemma, it suffices to show that 2) $\Longrightarrow$ 1).  Let $g \in A$ such that 
\begin{equation} \label{eq:3}
[b_1,g] = [z,a]
\end{equation}
for some $g \in A$, $g \neq 1$ and $z \in N$. We need to show that $g=a^\gamma$ for some $\gamma \in \Z$. 

Note first that in a free abelian group $A$ there is an automorphism $\theta \in \mathrm {Aut} A$ that maps $a$ to $a_1^\beta$ for some positive $\beta \in \N$. Replacing the basis $\{a_1, \ldots, a_m\}$ of $A$ with basis $\{\theta(a_1), \dots, \theta(a_m)\}$, if necessary, we may assume that $a = a_1^\beta$.

If $z = 1$ then $g=1 = a^0$. Suppose $z \neq 1$ and write $z$ as  $z = \prod_i b_i^{Q_i}$ for some Laurent polynomials $Q_i \in \Z A$. Then  the equality (\ref{eq:3}) above takes the form 
$$
b_1^{g-1} = \Pi_i b_i^{Q_i(a_1^\beta-1)}.
$$
Since $b_1, \ldots,b_n$ is a basis of the free module $N$, one gets $Q_i = 0$ for all $i >1$. Hence 
$$
g-1 = Q_1(a_1^\beta-1)
$$
 in the ring $\Z A$. Since $g \in A$ it is a product $g = \Pi_{i =1} ^m a_i^{\gamma_i}$ for some $\gamma_i \in \Z$, we can write $g = \frac{g_1}{g_2}$, where $g_1$ is a product of all $a_i^{\gamma_i}$ with  $\gamma_i >0$ and $g_2$ is a product of all $a_i^{|\gamma_i|}$ with  $\gamma_i < 0$.  Note that $g_1$ and $g_2$ do not have common variables. 
 Now,
 $$
 \frac{g_1-g_2}{g_2} = Q_1(a_1^\beta-1)
 $$
and therefore $Q_1 = \frac{P_1}{g_2}$, where $P_1$ is an ordinary polynomial from $\Z[a_1, \ldots,a_m]$. Hence, we have an  equality 
\begin{equation} \label{eq:4}
    g_1 -g_2 = P_1(a_1^\beta-1)
\end{equation}
in the ring of polynomials $\Z[a_1, \ldots,a_m]$. Consider a monomial order in monomials of $\Z[a_1, \ldots,a_m]$, for example, the deglex order where monomials are compared first by their degree and if they have the same degree then with respect to a lexicographical ordering. Write $P_1$ as a sum of terms $M_j$ (i.e., monomials with integer coefficients)
 $$
 P_1 = M_1 + \ldots + M_k
 $$
 where $M_1 > M_2 > \ldots M_k$ in the deglex order. Then the equation (\ref{eq:4}) becomes
 \begin{equation} \label{eq:5a}
     g_1-g_2 = M_1a_1^\beta -M_1 +M_2a_1^\beta - M_2 + \ldots + M_ka_1^\beta -M_k.
 \end{equation}
 Clearly, $M_1a_1^\beta$ is the leading term on the right and $-M_k$ is the smallest one, so they both cannot cancel out on the right-hand side. Then, assuming $g_1 > g_2$  we have $g_1 = M_1a_1^\beta$ and $g_2 = M_k$. Observe, that $g_1$ contains $a_1$ and $g_2$ does not.
  Now, the equality (\ref{eq:5a}), after canceling equal terms, turns into  
 $$
 M_1 - g_2 = M_2a_1^\beta - M_2 + \ldots -M_k.
 $$
which again has the  form (\ref{eq:5a}). By induction, $M_i = M_{i+1}a_1^\beta$ for all $i = 1, \ldots,k-1$. In particular, $g_2$ divides $g_1$. Since $g_1$ and $g_2$ do not have common divisors besides $1$ and $-1$, we deduce that $g_2 = \pm1$, and from $g = \frac{g_1}{g_2}$ we have $g_2 = 1$. It follows that $P_1 = a_1^{\beta(k-1)} + \ldots a_1^\beta +1$ and $g = g_1 = a^{k\beta}$ for some $k >0$. 
However, if we assume above that $g_2$ contains $a_1$. Then a similar argument shows that in this case $g_1 = 1$ and $g_2 = a_1^{k\beta}$ for some $k>0$, so $g = a_1^{-k\beta}$, as claimed.

\end{proof}

The two distinct yet equivalent conditions (2) and (3) in Lemma \ref{le:4} make it possible to define the cyclic subgroup $\langle a \rangle$ of $G$ for any $a \in A, a \neq 1,$ using different formulas. This flexibility is particularly useful when investigating the decidability of equations in $G$ or addressing other problems where the complexity of the defining formulas plays a crucial role. We summarize this observation in the following corollary.
\begin{cor} \label{co:1} The following holds in $G$:
 
    \begin{enumerate}
        \item There is a $\forall \exists$ positive formula $Cyc_A(x,y)$ of the group theory language such that for any $a \in A$ the formula $Cyc_A(a,y)$ defines the cyclic subgroup $\langle a\rangle$ in $G$.
\item There is a Diophantine formula $Cyc_A(x,y,b_1)$ with the parameter $b_1$ such that for any $1\neq a \in A$ the formula $Cyc_A(a,y,b_1)$ defines the cyclic subgroup $\langle a\rangle$ in $G$.

        \end{enumerate}
\end{cor}
\begin{proof}
    It follows directly from Lemma \ref{le:4} and Lemma \ref{le:3}, the latter is needed to define $N$ in $G$ without parameters.
\end{proof}

\begin{lemma} \label{le:basis}
There is a formula $Base_{top}(x_1, \ldots,x_m)$ of group language (without parameters) such that for any $c_1, \ldots,c_m \in G$   $Base_{top}(c_1, \ldots,c_m)$ holds in $G$ if and only if the tuples $(c_1, \ldots,c_m)$ and $(a_1, \ldots,a_m)$ are automorphically equivalent in $G$, i.e., $(c_1, \ldots,c_m)$ is a top basis of $G$.
\end{lemma}
\begin{proof}
    By Corollary \ref{co:1} the formula $Cyc_A(a_i,y)$ defines in $G$ the cyclic subgroup $\langle a_i\rangle $. It follows that the tuple $(a_1, \ldots,a_m)$ satisfies all the conditions below:

\begin{enumerate}
    \item [1)] $a_i \in G \smallsetminus N$ and $[a_i,a_j] = 1$ for all $1 \leq i,j \leq m$.
    
    \item [2)] The set $\{g \in G \mid G\models Cyc_A(a_i,g)\}$ is an abelian subgroup of $G$, containing $a_i \neq 1$. Denote this subgroup by $H_i$.
    \item [3)] $[H_i,H_j] = 1$, so $H = H_1\cdot \ldots \cdot H_m$ is an abelian group.
    \item [4)]  $H_i \cap H_{j_1} \cdot \ldots \cdot H_{j_{m-1}} = 1$ for any $1\leq j_k \leq m, j_k \neq i$.
    \item [5)] $H \cap N = 1$.
    \item [6)] $G = HN$.
\end{enumerate}
It is clear that each one of these conditions can be written by a formula of group language on elements $a_1, \ldots,a_m$. Replacing in this formula every occurence of the  element $a_i$ by a variable $x_i$ we obtain a formula of group language that we denote by $Base_{top}(x_1, \ldots,x_m)$. 

Observe, that $G \models Base_{top}(a_1, \ldots,a_m)$, in fact, this is true for every basis $(a_1', \ldots,a_m')$ of $A$, because from the beginning $(a_1, \ldots,a_m)$  was chosen as an arbitrary basis of $A$. Suppose that $G \models Base_{top}(c_1, \ldots,c_m)$ for some elements $c_1, \ldots,c_m \in G$. Then for every $c_i$ the formula $Cyc_A(c_i,y)$ defines in $G$ an abelian subgroup $C_i$, which contains a non-trivial element  $c_i$.  These subgroups satisfy all the conditions 1)-6), hence 

\begin{enumerate}
\item [7)] $C = C_1 \ldots C_m $ is an abelian subgroup of $G$. 
\item [8)] $C \simeq C_1 \times \ldots C_m \simeq A$.
\item [9)] $G = N \rtimes C$.
\end{enumerate}

From 7) it follows that $C$ is a free abelian of rank $m$, so every group $C_i$ is cyclic. Denote by $d_iu_i$, where $d_i \in A, u_i \in N$, a generator of $C_i$. It is easy to see that $d_1, \dots,d_m$ form a basis of $A$. Therefore, there is an automorphism $\psi \in Aut(A)$ such that $\psi(a_i) = d_i, i = 1, \ldots,m$. We claim that the automorphism $\psi$ extends to an automorphism $\hat \psi$ of $G$. Clearly, $\psi$ extends linearly to an automorphism $\lambda_\psi$ of $\Z A$.  Now, $\lambda_\psi$ gives rise to a semilinear bijection $\tilde \psi$ of the $\Z A$-module $N$, where for $u = \Pi b_i^{P_i} \in N, P_i \in \Z A$, 
$$
{\tilde \psi}(u)  = \Pi b_i^{\lambda_\psi(P_i)}.
$$
Now one can define an a map $\hat \psi$ on $G$ by $\hat \psi (au) = \psi(a){\tilde \psi}(u)$, where $a \in A, u \in N$. Clearly, ${\tilde \psi}$ is a bijection on $G$. Checking that it is a homomorphism  of $G$ is straightforward if using the formula (\ref{eq:mult-G}) of the multiplication in $G$.

The isomorphism ${\tilde \psi}^{-1}$ maps $d_iu_i$ to $a_iv_i$, where $v_i  =  {\tilde \psi}^{-1}(u_i) \in N$, so $(a_1v_1, \ldots,a_mv_m)$ forms a basis of a free abelian group $A'$, which is isomorphic to $A$ via a map $\theta: a_i \to a_iv_i, i = 1, \ldots,n$. Note, that the elements $a_iv_i$ act by conjugation on $N$ precisely the same way as elements $a_i$, hence the map $au \to \theta(a)u$, for $a \in A, u \in N$ is an isomorphism of $G$. This shows that the tuples $(a_1, \ldots,a_m)$,  $(a_1,v_1, \ldots, a_mv_m)$, and  $(d_1u_1, \ldots,d_mu_m)$ are automorphically equivalent in $G$. 

Recall that $c_i \in C_i = \langle d_iu_i\rangle$, so $c_i = (d_iu_i)^{\gamma_i}$, for some $\gamma_i \in \Z$. Now we apply Lemma \ref{le:4} where the basis $a_1, \ldots, a_m$, hence the group $A$, is replaced with the basis $(d_1u_1, \ldots,d_mu_m)$ and the group $A'$, and the element $a$ is replaced by $c_i \in A'$. This gives us that the formula $Cyc(c_i,y)$ defines in $G$ the cyclic subgroup $\langle c_i \rangle$ which is equal to $C_i$. Therefore $c_i = (d_iu_i)^{\pm 1}$, so there is an automorphism of $G$ that maps $c_i$ to $d_iu_i$ for $i = 1, \ldots,m$. This implies that the tuples $(c_1, \ldots,c_m)$ and $(a_1, \ldots,a_m)$ are automorphically equivalent in $G$, as required.

\end{proof}

\begin{lemma} \label{le:5}
   Let $1\neq a \in A$ and  $k,\ell \in \Z$. Then the following conditions are equivalent:
   \begin{enumerate}
       \item [1)] $k$ divides $\ell$,
       \item [2)] there is $z \in N$ such that $[b_1,a^k] = [z,a^\ell]$,
   \item [3)] for any $u \in N$ there exists  $z\in N$ such that $[u,a^k] = [z,a^\ell].$
   \end{enumerate}
\end{lemma}
\begin{proof}
    The argument is similar to the one in Lemma \ref{le:4}.
\end{proof}

\begin{cor} \label{co:2} The following holds in $G$:
 
    \begin{enumerate}
        \item There exists an $\forall \exists$ and positive formula $Div_A(x,y_1,y_2)$ of the group theory language such that for any $a \in A$  and any $k,\ell \in \Z$, the formula $Div_A(a,a^k,a^\ell)$ holds in $G$ if and only if $k$ divides $\ell$. 
\item For any $1\neq a \in A$  and any $k,\ell \in \Z$ a system of equations
\begin{equation} \label{eq:6}
[z,b_1] =1 \wedge [b_1,a^k] = [z,a^\ell]
\end{equation}
with parameters $a,b_1$ has a solution $z$ in $G$ if and only if $k$ divides $\ell$. 
 \end{enumerate}
\end{cor}
\begin{proof}
   The first statement follows from item 3) in Lemma \ref{le:5}, and the second statement follows from item 2) in Lemma \ref{le:5}.  
\end{proof}

\begin{prop}
Let $a \in A, a\neq 1$. Then the ring $\Z$ is interpretable in $G$ with the parameter $a$.
\end{prop}
\begin{proof}
We interpret $\Z$ in $G$ on the cyclic subgroup $\langle a\rangle=\{a^k \mid k \in \Z\}$. By Corollary \ref{co:1} the set $Z_a = \langle a\rangle$ is definable in $G$ with the parameter $a$. We interpret addition $\oplus$ on  $Z_a$ by the multiplication in $G$, since $a^k\cdot a^\ell = a^{k+\ell}$ the map $\mu_a: a^k \to k$ gives rise to an isomorphism 
$\mu_a: \langle Z_a;\oplus\rangle  \to \langle \Z;+\rangle$. 

Next, we interpret the integer division $\mid$ in $Z_a$ by a binary predicate $\mid^*$ which is defined either by the formula $Div_A(a,y_1,y_2)$ of Corollary \ref{co:2} or by the system of equations (\ref{eq:6}) of Corollary \ref{co:2}. The map $\mu_a$ gives an isomorphism of the structures 
$$
\mu_a: \langle Z_a;\oplus, \mid^* \rangle  \to \langle \Z;+,  \mid\rangle,
$$
so we interpreted $\langle \Z;+,  \mid\rangle$ in $G$. 

Finally, there is an argument due to Robinson (see \cite{Robinson}) that shows that the integer multiplication $\cdot$ is definable in $\langle \Z;+,  \mid\rangle$ by some formula $Mult(x_1,x_2,x_3)$.  Using formulas for $\oplus$ and $\mid^*$ one can rewrite the formula $Mult(x_1,x_2,x_3)$ into a formula $M_a(x_1,x_2,x_3,a)$ with parameter $a$ which defines a multiplication $\otimes$ on $Z_a$ in such a way that the map $\mu_a$ gives an isomorphism 
$$
\mu_a: \langle Z_a;\oplus, \otimes \rangle  \to \langle \Z; +,  \cdot \rangle.
$$
This gives an interpretation $\Z_a = \langle Z_a;\oplus, \otimes \rangle $ 
of $\Z$ in $G$ with the parameter $a$.
\end{proof}

\begin{remark}
We want to mention the following about the interpretations of $\Z$ in $G$:
  \begin{enumerate}
      \item For any $a\in A, a \neq 1$, the interpretation $\Z_a$ is \emph{injective}, that is, we use the standard equality $=$ from $G$ to define the equality in $\Z_a$. This might be useful for applications, in particular, in studying the elimination of imaginaries. 

      \item the structure $\langle \Z;+,  \mid\rangle$ is interpretable in $G$ by Diophantine formulas. 
  \end{enumerate}  
\end{remark}

Now we show that the isomorphisms $\Z_a \to \Z_d$ are uniformly definable in $G$ with parameters $a, d \in A$.

\begin{lemma} \label{le:6}
Let $a, d \in A, a, d \neq 1$ and  $k,l \in \Z$. Then $k = l$ if and only if for some nontrivial $u \in N$ (equivalently, every $u \in N$) the following holds:
$$[u,a, d^k]=[u,a^\ell,d] \  \ mod \ G_4,$$ 
where $G_4$ is the fourth term of the lower central series of $G$. 
 \end{lemma}
\begin{proof}
    The ternary left-normed commutator $[x_1,x_2,x_3]$ defines a map 
    $N\times G \times G \to G_3/G_4$ which is linear in every variable.
    If $u \in N, u \neq 1$, then $[u,a,d] \neq 1$. Now $[u,a,d^k] = [u,a,d]^k \  \ mod \ G_4$ and     $[u,a^\ell,d] = [u,a,d]^\ell \  \ mod \ G_4$. Hence 
    $$
    [u,a, d^k]=[u,a^\ell,d] \  \ mod \ G_4,
    $$
   implies that 
$$
[u,a,d]^k  = [u,a,d]^\ell \  \ mod \ G_4,
$$
and therefore   $k = \ell$ since $G/G_4$ is torsion-free by Lemma \ref{le:1}. 
   
\end{proof}

\begin{cor} \label{co:3}
    There is a formula $Iso_\Z(x_1,x_2,y_1,y_2)$ in the group theory language such that for any $a,d \in A, a,d \neq 1,$ the formula $Iso_\Z(x_1,x_2,a,d)$ with parameters $a,d$ defines in $G$ the graph of the isomorphism $\Z_a \to \Z_d$.
\end{cor}

\begin{remark}
    Corollary \ref{co:3} allows one to fix an arbitrary $a\in A, a\neq 1$, make some computations with integers in $\Z_a$ and then transfer the result through $Iso_\Z(x_1,x_2,a,d)$ into $\Z_d$. For example, the set $$
    \{(a^k,d^{2^k}) \mid k \in \N\}
    $$
  is uniformly definable in $G$ with parameters $a,d \in A, a,d \neq 1$.

\end{remark}

\begin{cor} \label{co:4}
There is a formula $Prod_{A,m}(x_1,x_2,\bar y, \bar z)$ of the group language, here $\bar y, \bar z$  are tuples of variables of length $m$,  such that for any $a \neq 1, d_1, \ldots,d_m \in A$, $a^{\gamma_1}, \ldots, a^{\gamma_m } \in \Z_a$, and $g \in G$  the following holds: 
$$
G \models Prod_A(g,a, \bar d,  {\bar a}^{\bar \gamma}) \Longleftrightarrow g = d_1^{\gamma_1} \ldots d_m^{\gamma_m}
$$
where $\bar d = (d_1, \ldots,d_m), {\bar a}^{\bar \gamma}  = (a^{\gamma_1}, \ldots, a^{\gamma_m } )$. 

In particular, the formula $Exp_A(x_1,x_2,y_1,z_1) = Prod_{A,2}(x_1,x_2,\bar y, \bar z)$  defines in $G$ the $\Z$-exponentiation on $A$, that is for any $a, d \neq 1 \in A$ and any $\gamma \in \Z$, for any $g \in G$ the following holds: 
$$
G \models Prod_A(g,a, d,  a^\gamma) \Longleftrightarrow g = d^\gamma.
$$
    
\end{cor}

\begin{cor} \label{co:5}
There is a formula $Base_A(x_1,x_2,\bar y, \bar z)$ of the group language such that for any $d_1, \ldots,d_m \in A$ 
$$
G \models Base_A(d_1, \ldots,d_m) \Longleftrightarrow  (d_1, \ldots,d_m)  \text{ is a basis of}  A 
$$

\end{cor}

\section{Definability of $\Z$- and $\Z A$-exponentiations in $G$}\label{se:4}

\subsection{Definability of $\Z$-exponentiations in $G$}

We first prove some auxiliary results.

\begin{lemma} \label{le:7}

  If $a\in A, a\neq 1, u\in N, u\neq 1,$  and $k\in\mathbb Z$. Then the following holds:
\begin{enumerate}
    \item [1)]  $(au)^k  = a^kw$, where $w =  u^{a^{k-1} + a^{k-2} + \ldots +1}$
    \item [2)] this $w = u^{a^{k-1} + a^{k-2} + \ldots +1} $ is the unique solution of the equation $[u,a^k]=[w,a]$.
\end{enumerate}
 \end{lemma} 
\begin{proof}
   Let $k > 0$.  Then by induction 
   $$
 (au)^k = (au)^{k-1}au  = a^{k-1}u^{a^{k-2} + \ldots +1 } au = a^ku^{a^{k-1} + \ldots +1}, 
 $$  
 which proves 1). 
  To see 2) rewrite the equality $[u,a^k]=[w,a]$ in the form $u^{a^k -1} = w^{a-1}$. Dividing $a^k-1$ by $a-1$ one gets a unique solution $w = u^{a^{k-1} + a^{k-2} + \ldots +1}$.

\end{proof}

\begin{cor} \label{co:6}
  There is a formula $Exp_{mix}(x_1,x_2,x_3,y)$ of  group theory language such that  for any $a\neq 1 \in A$, any $u \neq 1 \in N$, and any $k \in \Z$ for any $g \in G$ the following holds:
$$
G \models Exp_{mix}(g,u,a^k,a) \Longleftrightarrow g = (au)^k.
$$
  
\end{cor}
\begin{proof}
It follows from Lemma \ref{le:7} and the results of the previous section.
    
\end{proof}
 
\begin{lemma} Let $u \in N$, $k\in\mathbb Z$.
 Then for every $a\in A$  there exists $v\in N$ such that \begin{equation} \label{eq:lemma-9-1}
  (au)^k=a^ku^k[v,a].   
 \end{equation}
 Moreover, $w= u^k$ is the only element in $N$ such that for every $a \in A$ the equation  
 \begin{equation} \label{eq:lemma-9-2}
 (au)^k=a^kw[v,a]
 \end{equation}
 has a solution $v \in N$.

 \end{lemma} 
\begin{proof}
    We prove first, by induction on $k$,  that 
\begin{equation} \label{eq:lemma-9-3}
u^ka = au^k[v,a]    
\end{equation}
for  some $v \in N$. For $k = 1$ one has  $ua = au[u,a]$. By induction  
$$
u^{k+1}a = u^kua = u^kau[u,a] = au^k[v_1,a]u[u,a] = au^{k+1}[v_1,a][u,a]=au^{k+1}[v_1u,a],
$$
and (\ref{eq:lemma-9-3}) follows.

Now we prove (\ref{eq:lemma-9-1}) by induction on $k$ as follows (here an element $v_2 \in N$ appears as the result of  the induction step and an element $v_3$ as an application of the (\ref{eq:lemma-9-3})):
$$
(au)^{k+1} = (au)^kau  = a^ku^k[v_2,a]au = a^ku^ka[v_2,a][[v_2,a],a]u = 
$$
$$
= a^kau^k[v_3,a][v_2,a][[v_2,a],a]u = a^{k+1}u^{k+1}[v_3v_2[v_2,a],a] =  a^{k+1}u^{k+1}[v,a], 
$$
where $v = v_3v_2[v_2,a],a]$.

For the uniqueness, let $u\in N, w\in G$. If for any $a\in A$, there exists $v_1\in N$ such that 
$$(au)^k=a^kw[v_1,a].$$
Then by (\ref{eq:lemma-9-1}) there exists $v\in N$ such that 
$$a^ku^k[v,a]=a^kw[v_1,a].$$
Hence
$$w^{-1}u^k=[v_1,a][v^{-1},a]=[v_1v^{-1},a].$$
Since $N$ is a free $\Z A$-module, one gets $w^{-1}u^k=1.$ Indeed, let $w^{-1}u^k = \Pi_i b_i^{P_i}$ and $v_1v^{-1} = \Pi_i b_i^{Q_i}$ for some $P_i, Q_i \in \Z A$. Then  for every $i$ $P_i = Q_i(a-1)$. It follows that $a-1$ divides $Q_i$ for every $a \in A$. In particular, every polynomial $a_1^\ell -1, \ell > 0$ divides every $P_i$. Hence, $P_i = 0$ for all $i$, so $w^{-1}u^k = 1$, as claimed.
\end{proof} 

\begin{cor} \label{co:7}
There is a formula $Exp_N(x_1,x_2, x_3, y)$ of  group theory language such that  for any $a\neq 1 \in A$, any $u \neq 1 \in N$, and any $k \in \Z$, for any $g \in G$ the following holds:
$$
G \models Exp_N(g,u,a^k,a) \Longleftrightarrow g = u^k.
$$
  
\end{cor}

 \begin{theorem} \label{th:2}
     There is a formula $Exp_G(x_1,x_2,x_3,y)$ of the group theory language such that  for any $a \neq 1 \in A,  g,h\in G, k\in\mathbb Z$
     $$
     G \models Exp_G(g,h,a^k, a) \Longleftrightarrow g^k=h.
     $$
 \end{theorem}
 \begin{proof}
     It follows from Corollaries \ref{co:4}, \ref{co:6}, and \ref{co:7}.
 \end{proof}
 
 \subsection{Interpretation of the $\Z A$-action on $N$ in $G$}

In this section, we show that the standard action of $\Z A$ on $N$ is interpretable in $G$ with parameters $a_1, \ldots, a_m, b_1, \ldots, b_n$. We fix an element $a = a_1 \in A$ and use the interpretation $\Z_a$ of the ring of integers $\Z$ in $G$.

 For a tuple $\bar \alpha  = (\alpha_1, \ldots,\alpha_m) \in \Z^m$  denote by $\lambda_{\bar \alpha}$ the homomorphism $\lambda_{\bar \alpha}: \Z[a_1, \ldots,a_m] \to \mathbb Z$ such that $a_i \to \alpha_i, i = 1, \ldots,m$.  The kernel $I_{\bar \alpha}$ of $\lambda_{\bar \alpha}$ is the  ideal generated in  $\Z[a_1, \ldots,a_m] $  by  $\{a_1- \alpha_1, \ldots, a_m -\alpha_m\}$.
Notice, that for every polynomial $P = P(a_1, \ldots,a_m) \in \Z[a_1, \ldots,a_m]$ one has $\lambda_{\bar \alpha}(P) = P(\alpha_1, \ldots,\alpha_m)$, so 
$$
P(a_1, \ldots,a_m) = P(\alpha_1, \ldots,\alpha_m)  + \Sigma_{i=1}^m(a_i-\alpha_i)f_i,
$$
for some $f_i \in \Z[a_1, \ldots,a_m]$.

Let $C$ and $D$ be rings and $\Lambda$ a set of homomorphisms from $C$ into $D$.  Recall that $C$ is discriminated into $D$ by a set $\Lambda$ if for any finite subset $C_0 \subseteq C$ there is a homomorphism $\lambda \in \Lambda$ which is injective on $C_0$.  

\begin{lemma} \label{claim1} The ring $\Z[a_1, \ldots,a_m]$ is discriminated into $\Z$ by the set of homomorphisms $\{\lambda_{\bar \alpha} \mid \bar \alpha   \in \Z^m\}$. 
\end{lemma}
\begin{proof}
    By induction on $m$.
\end{proof}

 Denote by $(N)^{I_{\bar \alpha}}$ the submodule of the $\Z A$-module  $N$ generated by $N^{I_{\bar \alpha}}$. Observe, that every element $u \in N^{I_{\bar \alpha}}$, as well as in $(N)^{I_{\bar \alpha}}$ admits a decomposition of the type $ u = \Pi_i b_i^{U_i}$, where $U_i \in I_{\bar \alpha}$. As an abelian group, $(N)^{I_{\bar \alpha}}$  is generated by the set $\{g^Q \mid g \in N, Q \in I_{\bar \alpha}\}$, hence by the set $\{g^{a_i - \alpha_i} \mid g \in N, i = 1, \ldots, m\}$. The result below shows that    the  subgroup $(N)^{I_{\bar \alpha}}$ is definable in $G$ uniformly in $a_1, \ldots,a_m$ and $(\alpha_1, \ldots,\alpha_m)$. Of course, integers are not elements of $G$, so instead of $(\alpha_1, \ldots,\alpha_m)$ we use their interpretations $a^{\alpha_1}, \ldots, a^{\alpha_m}$ in $\Z_a$.
 
\begin{lemma}
   There is a formula $\phi(x,\bar y, \bar z)$ of the group language such that for any basis $(a_1, \ldots,a_m)$ of $A$ and any tuple $(\alpha_1, \ldots,\alpha_m) \in \Z^m$ the formula 
    $$
    \phi(x,a_1, \ldots,a_m, a^{\alpha_1}, \ldots, a^{\alpha_m})
    $$
    defines in $G$ the subgroup $(N)^{I_{\bar \alpha}}$.
\end{lemma} 
 \begin{proof}
     Indeed, every element $u \in (N)^{I_{\bar \alpha}}$ can be written in the form 
     $$
     u = g_1^{a_1 -\alpha _1}\ldots  g_m^{a_m -\alpha _m} = g_1^{a_1}g_1^{-\alpha _1}\ldots  g_m^{a_m}g_m^{-\alpha _m}, 
     $$
     where $g_1,\ldots ,g_m\in N$. The elements $g_i^{\alpha_i}$ are the conjugations $g_i^{a_i} = a_i^{-1}g_ia_i$, so they are defined by formulas in $G$ with parameters $a_i$, $i = 1, \ldots,m$. Exponents $g_i^{-\alpha _i}$ can be written by formulas in $G$ using the exponentiation formula $Exp_A(x_1,x_2,y_1,z_1)$ from Corollary \ref{co:4}.

 \end{proof}

\begin{lemma} \label{le:11}
Let $g, h \in N$ and $P\in \Z[\bar a]$. Then $g^P = h$ if and only if for every $\alpha_1, \ldots \alpha_m  \in \Z$ the following condition holds:
\begin{equation} \label{eq:g-h-P}
g^{P(\alpha_1, \ldots,\alpha_m)} = h \ mod \ (N)^{I_{\bar \alpha}}.
\end{equation}
\end{lemma}
\begin{proof} 
Let $g^p = h$. Fix an arbitrary $\bar \alpha = (\alpha_1, \ldots,\alpha_m) \in \Z^m$.
As we mentioned above 
$$
P(a_1, \ldots,a_m) = P(\alpha_1, \ldots,\alpha_m) + U 
$$
for some $U \in I_{\bar \alpha}$. Hence 
$$
h = g^P = g^{p(\bar \alpha) +U}= g^{P(\bar \alpha)} g^U
$$
so the condition (\ref{eq:g-h-P}) holds.

Conversely, suppose for  $g, h \in N$, $P \in \Z[\bar a]$, and arbitrary $\bar \alpha = (\alpha_1, \ldots,\alpha_m) \in \Z^m$  the condition (\ref{eq:g-h-P}) holds. We mentioned above, that $P = P(\bar \alpha) \ mod \ I_{\bar \alpha}$,  so (\ref{eq:g-h-P}) implies that 
\begin{equation} \label{eq:8}
g^P = h u
\end{equation}
for some $u \in (N)^{I_{\bar \alpha}}$. We need to show that $g^P = h$. 

Decompose elements $g, h, u$ with respect to the basis $b_1, \ldots,b_n$:

$$
g = \Pi_i b_i^{R_i}, \ \ h =  \Pi_i b_i^{S_i}, \ \ u = \Pi_i b_i^{U_i}, 
$$
where $R_i, S_i$ are Laurent polynomials from $\Z A$ and $U \in I_{\bar \alpha}$. Then the equality \ref{eq:8} implies 
\begin{equation} \label{eq:9}
\Pi_i b_i^{R_iP} =  \Pi_i b_i^{S_i +U_i}
\end{equation}
Since $N$ is free with basis $b_1, \ldots,b_n$  it follows from  (\ref{eq:9})  that for $i=1,\ldots ,n$ 
\begin{equation} \label{eq:10}
R_i P = S_i +U_i
\end{equation}
in the ring $\Z A$.
 Let
 $$
 R_i = \frac{F_i}{K_i}, \ \ S_i = \frac{H_i}{L_i},
 $$
 where $F_i, H_i$ are ordinary polynomials from $\Z[\bar a]$, and 
$K_i, L_i$ are non-zero monomials from $\Z[\bar a]$. Now, it follows from (\ref{eq:10}) that 
$$
F_i P L_i =  H_i K_i + U_i K_i L_i
$$
in the polynomial ring $\Z [\bar a]$.
 
 Evaluating these polynomials at $a_1 = \alpha_1, \ldots ,a_m = \alpha_m$ one gets 
\begin{equation} \label{eq:11}
 F_i(\bar \alpha) P(\bar \alpha) L_i(\bar \alpha) =  H_i(\bar \alpha) K_i(\bar \alpha) 
\end{equation}
since $U_i(\bar \alpha)= 0$.
  Hence by Lemma \ref{claim1}, we have equality in the polynomial ring $\Z[\bar a]$:
 $$
 F_i P L_i =  H_i K_i,
$$
which implies that 
$$
\frac{F_i}{K_i} P = \frac{H_i}{L_i}
$$
in the ring of Laurent polynomials $\Z A$, that is 
$$
R_i P = S_i
$$
for every $i = 1, \ldots,n$. Hence  $g^P = h$, as claimed.
\end{proof}

Now we can prove that in the group $G$ the action of $\Z A$ on $N$ is definable. Again, Laurent polynomials $Q \in \Z A$ are not elements of $G$, so we first represent them by their codes $k =   \nu (Q)$, and then represent the number $k$ as $a^k$ in the interpretation $\Z_a$ of $\Z$ in $G$. As above, we may choose $a= a_1$. The result is stated as follows.

\begin{theorem} \label{th:3}
There is a formula $Act_{\Z A}(x_1,x_2, x_3,y_1)$  of group theory language such that for any $a \in A$, $g,h \in N$, $Q \in \Z A$, and $k=  \nu (Q)$ the following holds:
$$
g^Q = h \Longleftrightarrow G\models Act_{\Z A}(g,h, a^k,a).
$$

\end{theorem}
\begin{proof}
Fix $a \in A, a\neq 1$. We prove first, that the condition (\ref{eq:g-h-P}) from Lemma \ref{le:11} 
 can be written by formulas in $G$.  Let $P \in \Z[a_1, \ldots,a_n]$.  Given the $k = \nu (P) \in \Z$   and a tuple of numbers $\bar \alpha = (\alpha_1, \ldots, \alpha_m)$ one can compute the value $s = P(\bar \alpha) \in \Z$, so there is a formula in arithmetic $Val_\Z(x_1,x_2, \bar y)$ such that 
 $$
 s = P(\bar \alpha) \Longleftrightarrow \Z \models Val_\Z (\nu (P),s, \bar \alpha).
 $$
Hence, there is a formula $Val_G(x_1,x_2, \bar y, z)$ of group theory such that 
$$
a^s =a^{P(\bar \alpha)} \Longleftrightarrow G \models Val_G(a^{\nu (P)},a^s, a^{\alpha_1}, \ldots,a^{\alpha_m}, a)
$$
To prove this one needs to rewrite the formula $Val_\Z(x_1,x_2, \bar y)$ into a formula $Val_G(x_1,x_2, \bar y, z)$ using the interpretation $\Z_a$ in $G$.
The next step is to write a formula that for any $g, f \in N$, $P \in \Z[\bar a]$, and $\bar \alpha = (\alpha_1, \ldots, \alpha_m) \in \Z^m$ states that  $g^{P(\bar \alpha)} = f$. This can be done using the formula $Exp_G(x_1,x_2,x_3,y)$ from Theorem \ref{th:2}. Indeed, by Theorem \ref{th:2} 
$$
g^s = f \Longleftrightarrow G \models Exp_G(g,f,a^s,a).
$$
Now, given $a^{\nu (P)}$, and $a^s, a^{\alpha_1}, \ldots,a^{\alpha_m}, a$ the formula 
$$
Val_G(a^{\nu (P)},x_2, a^{\alpha_1}, \ldots,a^{\alpha_m}, a)
$$ 
holds in $G$ only for $x_2= a^{p(\bar \alpha)}$. Therefore,
$$
g^{P(\bar \alpha)} = f \Longleftrightarrow G \models \exists x_2 Val_G(a^{\nu (P)},x_2, a^{\alpha_1}, \ldots,a^{\alpha_m}, a) \wedge Exp_G(g,f,x_2,a).
$$
We denote such a formula by $Exp_{N,\Z A}$.

Now we are able to write down by a formula the condition (\ref{eq:g-h-P}) from Lemma \ref{le:11}. Indeed, for any $\bar \alpha = (\alpha_1, \ldots, \alpha_m) \in \Z^m$ the set $(N)^{I_{\bar \alpha}}$ is definable in $G$. The condition $g^{P(\bar \alpha)} = f$ can be expressed by formula $Exp_{N,\Z A}$, so there is a formula $Act_{\Z [\bar a]}(x_1,x_2, x_3,y_1)$ of group language such that for any $a \in A$, $g,h \in N$, $P \in \Z [\bar a]$, and $k=  \nu (P)$ the following holds:
$$
g^P = h \Longleftrightarrow G\models Act_{\Z [\bar a]}(g,h, a^k,a).
$$
This defines the action of the ring of ordinary polynomials $\Z[\bar a]$ on $N$.

The last step is to get a formula $Act_{\Z A}(x_1,x_2, x_3,y_1)$ that defines the action of the ring of the Laurent polynomials  $\Z A$ on $N$. 
Let $Q \in \Z A$. Then $Q = \frac{P}{M}$, where $P,M \in \Z[\bar a]$ and $M$ is a monomial. In this case, $g^Q = h$ if and only if $g^P =h^M$, and the exponentiation here by the polynomials $P$ and $M$, so the formula   $Act_{\Z [\bar a]}(x_1,x_2, x_3,y_1)$ does it for us.

This proves the theorem.

\end{proof}

\begin{cor}
    
\label{co:8}
    There is a formula $\Lambda_N(x, \bar z, \bar t, \bar s)$ of group language such that for any $u \in N$, any top basis $\bar c = (c_1, \ldots, c_m) $ of $G$, any tuple $\bar u = (u_1, \ldots,u_n)$ of elements of $N$, any tuple  $\bar P = (P_1, \ldots,P_n) \in \Z [c_1^{\pm 1}, \ldots,c_m^{\pm 1}]$ one has 
\begin{equation} \label{eq:12}
u =  u_1^{P_1}  \ldots u_n^{P_n} \Longleftrightarrow G \models \Lambda _N(u,c_1^{\nu(\bar P)}, \bar c, \bar u),
\end{equation}
where  $c_1^{\nu(\bar P)} = (c_1^{\nu(P_1)}, \ldots,c_1^{\nu(P_n)})$.
\end{cor} 

\begin{proof}
    Writing down the formula $\Lambda_N(x,\bar z, \bar t, \bar s)$ suffices to define the powers $u_j^{P_j}$ by formulas in $G$. If $c_1 = a_1, \ldots,c_m = a_m$ then this is done in Theorem \ref{th:3}. Since for the top basis $\bar c$ there an automorphism $\phi \in Aut(G)$ which maps $\bar a$ to $\bar c$ component-wise, it follows that the result holds for $\bar c$ as well.
\end{proof}

\begin{cor} \label{co:9}
    There is a formula $\Lambda_G(x,\bar y, \bar z, \bar t, \bar s)$ of group language such that for any $g \in G$, any top basis $\bar c = (c_1, \ldots, c_m) $ of $G$, any tuple $\bar u = (u_1, \ldots,u_n)$ of elements of $N$, any tuple $\bar \gamma = (\gamma_1, \ldots,\gamma_m) \in \Z^m$, and any tuple $\bar P = (P_1, \ldots,P_n) \in \Z [c_1^{\pm 1}, \ldots,c_m^{\pm 1}]$ one has 
\begin{equation} \label{eq:12}
g = c_1^{\gamma_1} \ldots c_m^{ \gamma_m} u_1^{P_1}  \ldots u_n^{P_n} \Longleftrightarrow G \models \Lambda_G (g,c_1^{\bar \gamma}, c_1^{\nu(\bar P)},\bar c, \bar u),
\end{equation}
where $c_1^{\bar \gamma}  = (c_1^{\gamma_1}, \ldots, c_1^{\gamma_m})$ and $c_1^{\nu(\bar P)} = (c_1^{\nu(P_1)}, \ldots,c_1^{\nu(P_n)})$.
\end{cor} 

\begin{proof}
   The argument is similar to Corollary \ref{co:8}. The only difference is that to  define the powers $c_i^{\gamma_i}$ by formulas in $G$ one uses  Theorem \ref{th:2}. 
\end{proof}

\begin{theorem} \label{th:4}
    There is a formula $Base_G$ of group theory language (without parameters) which defines in $G$ the set  of all bases of $G$.
\end{theorem}
Let $q = (c_1, \ldots,c_m,u_1, \ldots,u_n)$,  where $c_i \in G\smallsetminus N$ and $u_i \in N$. Denote $\bar c = (c_1, \ldots,c_m)$ and  $\bar u = (u_1, \ldots,u_n)$. By Definition \ref{de:bases} $q$ is a basis of $G$ if and only if the following conditions hold:
\begin{enumerate}
    \item [1)] $C = \langle c_1, \ldots,c_m\rangle \simeq A$,
    \item [2)] $G = C \cdot N$,
    \item [3)] $N$ is a free $\Z[c_1^{\pm 1}, \ldots,c_m^{\pm 1}]$-module with basis $u_1, \ldots,u_n$.
\end{enumerate}

It suffices to describe a formula $Base_G(\bar x,\bar y)$ of the group theory language without parameters such that $G \models Base_G(\bar c,\bar u)$ if and only if $q$ satisfies the conditions 1)-3) above. 

By Lemma \ref{le:basis} there is a formula $Base_{top}(x_1, \ldots,x_m)$ of group language (without parameters) such that for any $c_1, \ldots,c_m \in G$   $G \models Base_{top}(c_1, \ldots,c_m)$ if and only if $(c_1, \ldots,c_m)$ is a top basis of $G$, in which case $c_1, \ldots,c_m$ satisfy the conditions 1) and 2). 

By Corollary \ref{co:8} for $u \in N$ and $P_1, \ldots,P_n \in \Z[c_1^{\pm 1}, \ldots,c_m^{\pm 1}]$ the condition $u =   u_1^{P_1}  \ldots u_n^{P_n} $ from  (\ref{eq:12}) is described by the formula $\Lambda_N$. 
Using this one can easily write a formula that states that the element $u$ such that $u =   u_1^{P_1}  \ldots u_n^{P_n} $ is unique in $N$, which means that $N$ is a free $\Z[c_1^{\pm 1}, \ldots,c_m^{\pm 1}]$-module with basis $u_1, \ldots,u_n$. 

This proves the theorem.

\section{Bi-interpretability of $G$ and $\Z$}\label{se:5}

Fix a tuple of parameters $p=(a_1,\ldots ,a_m,b_1,\ldots ,b_n)$ in $G$, where $\{a_1,\ldots ,a_m\}$ and $\{b_1,\ldots ,b_n\}$ are bases of $A$ and $B$ respectively. In this section we prove that $G$ and $\Z$ are bi-interpretable with parameters $\bar p$ in $G$.  

Recall that $\Z$ is interpreted in $G$ on the cyclic subgroup $\langle a_1\rangle$, we denoted this interpretation $\Z_{a_1} = \Gamma(G,\bar p)$ with the coordinate map $\mu_\Gamma : \Z_{a_1} \to \Z$ given by $a_1^\gamma \to \gamma$ for all $\gamma \in \Z$.  

Meanwhile, $G$ is interpreted in $\Z$ by tuples of integers 
$$(\gamma_1, \ldots, \gamma_m, \nu(P_1), \ldots,\nu(P_n)),$$
where $\gamma_i \in \Z$,  $P_i \in \Z[a_1^{\pm 1}, \ldots, a_m^{\pm 1}] = \Z A$, and $\nu(P_i)$ is the code of the Laurent polynomial $P_i$. For simplicity we assume that the enumeration $\nu$ is such that the set $\nu(\Z A) = \{\nu(P) \mid P \in \Z A\}$ is precisely $\Z$.   In this case the interpretation $G \simeq \Delta(\Z)$  of $G$ in $\Z$ is on the set $\Z^{m+n}$ with the coordinate map $\mu_\Delta: \Z^{m+n} \to G$ given by 
 $$
 (\gamma_1, \ldots, \gamma_m, \nu(P_1), \ldots,\nu(P_n)) \to g = a_1^{\gamma_1} \ldots a_m^{\gamma_m} b_1^{P_1} \ldots b_n^{P_n}.
 $$

\begin{theorem} \label{th:5} Let $G = B \wr A$.Then the following holds:
\begin{enumerate} 
\item [1)] Let  $p=(a_1,\ldots ,a_m,b_1,\ldots ,b_n)$ in $G$, where $\{a_1,\ldots ,a_m\}$ and $\{b_1,\ldots ,b_n\}$ are bases of $A$ and $B$ respectively. Then the group $G$ and the ring $\Z$ are bi-interpretable with each other through the interpretations $G \simeq \Delta(\Z)$ and $\Z \simeq \Gamma(G,\bar p)$.
\item [2)] The group $G$ and the ring $\Z$ are regularly bi-interpretable with each other through the interpretations $G \simeq \Delta(\Z)$ and $\Z \simeq \Gamma(G,Base_G)$, where $Base_G$ is the formula from Theorem \ref{th:4} that defines the set of all bases in $G$.
\end{enumerate}
\end{theorem}

\begin{proof}
    To prove 1) we need to show that the coordinate maps  $\mu_{\Gamma \circ \Delta}$ and $\mu_{\Delta \circ \Gamma}$ are definable, correspondingly in $G$ (with parameters $\bar p$) and $\Z$. 

    We consider $\mu_{\Gamma \circ \Delta}$, the other one is easy.
    Note that
    $$
    \mu_{\Gamma \circ \Delta} : (a_1^{\gamma_1},\ldots, a_1^{\gamma_m}, a_1^{\nu(P_1)}, \ldots,a_1^{\nu(P_n)}) \to g = a_1^{\gamma_1} \ldots a_m^{\gamma_m} b_1^{P_1} \ldots b_n^{P_n}.
    $$
    By Corollary  \ref{co:9} the formula $\Lambda_G(x,\bar x, \bar y, \bar z, \bar t, \bar s)$ does precisely what we need here, see (\ref{eq:12}).
    
    2) follows from 1) by Theorem \ref{th:4}.
    
\end{proof}

\begin{cor}  The restricted wreath product $\mathbb Z^n\wr\mathbb Z^m$ is QFA.
\end{cor}

\section{Condensed center-by-metabelian groups}

\subsection{The space of finitely generated marked groups.}
We begin by recalling some definitions from \cite{Gri}. For an abstract set $I$, we denote by $2^I$ the space of all subsets of $I$ endowed with the product topology (equivalently, the topology of pointwise convergence of the indicator functions). If $G$ is a group, we denote by $\mathcal N(G)$ the set of all normal subgroups of $G$ endowed with the subspace topology induced by the inclusion $\mathcal N(G) \subseteq 2^{G}$. Thus, the base of neighborhoods of $N\in \mathcal N(G)$ consists of the sets
\begin{equation}\label{U}
U(\mathcal F,N)=\{ M\le G \mid M\cap \mathcal F = N\cap \mathcal F\},
\end{equation}
where $\mathcal F$ ranges in the set of all finite subsets of $G$.  

Let $n\in \NN$. An \emph{$n$-generated marked group} (or simply a \emph{marked group}) is a pair $(G,A)$, where $G$ is a group and $(a_1, \ldots, a_n)\subseteq G^n$ is an $n$-tuple such that $G$ is generated by $\{a_1, \ldots, a_n\}$. Two $n$-generated marked groups  $(G,(a_1, \ldots, a_n))$ and $(H,(b_1, \ldots, b_n))$ are \emph{isomorphic} if the map sending $a_i$ to $b_i$ for all $i=1, \ldots, n$ extends to an isomorphism $G\to H$. The set of all isomorphism classes of $n$-generated marked groups is denoted by $\mathcal G_n$. Following the common practice, we abuse the notation and terminology by not distinguishing between marked groups and their isomorphism classes.  

For every $n\in \NN$, let $F_n$ be the free group of rank $n$ with a fixed basis $X=\{x_1, \ldots, x_n\}$. For every $(G, A)\in \G_n$, where $A=(a_1, \ldots, a_n)$, let $\e_{(G,A)}\colon F_n\to G$ be the homomorphism such that
$$
\e_{(G,A)}(x_i)=a_i \;\;\; \forall\, i=1, \ldots, n.
$$ 
The map $(G,A)\to \Ker \e_{(G,A)}$ is a well-defined bijection between $\G_n$ and $\mathcal N(F_n)$.  
The topology on $\G_n$ is defined as the pull-back topology under this map; that is, a sequence $\{(G_i, A_i)\}_{i\in \NN}$ converges to $(G,A)$ in $\G_n$ if and only if the sequence $\{ \Ker \e_{(G_i, A_i)}\}_{i\in \NN}$ converges to $\Ker \e_{(G,A)}$ in $\mathcal N(F_n)$.

\begin{example}
We have $\lim\limits_{n\to \infty} (\ZZ_n,\{ 1\})= (\ZZ,\{ 1\})$ in $\G_1$. 
\end{example}

It is easy to see that the map $(G, (a_1, \ldots, a_n))\mapsto (G, (a_1, \ldots, a_n, 1))$ defines a continuous embedding $\G_n \to \G_{n+1}$. 

\begin{definition}
The topological union $$\G=\bigcup_{n=1}^\infty \G_n$$ is called the \emph{space of finitely generated marked groups}. 
\end{definition}

By the definition of a topological union, a set $U\subseteq \G$ is open if and only if $U\cap \G_n$ is open in $\G_n$ for all $n\in \NN$. Alternatively, the topology on $\G$ can be as follows: a sequence $\{(G_i, A_i)\}_{i\in \NN}$ converges to a marked group $(G, A)\in \G_n$ if and only if there exists  $M \in \NN$ such that $(G_i, A_i)\in \G_n$ for all $i\ge M$, and the subsequence $\{(G_i, A_i)\}_{i=M}^\infty$ converges to $(G,A)$ in $\G_n$. 

Recall that a topological space is \emph{Polish} if it is separable and completely metrizable. For the result below, see \cite{Gri} and \cite[Proposition 3.7]{Osi21a}.

\begin{prop}
The space $\G$ is Polish.
\end{prop}

We will need the following elementary yet useful observation.

\begin{lemma}[{\cite[Lemma 2.2]{MOW}}]\label{Lem:NG}
Let $G$ be a group generated by a finite set $X$. Given a normal subgroup $N\lhd G$, we let $X_N$ denote the image of $X$ under the natural homomorphism $G\to G/N$. The map $\mathcal N(G)\to \G$ defined by the formula $N\mapsto (G/N, X_N)$ for all $N\in \mathcal N(G)$ is continuous.
\end{lemma}

\subsection{The main construction}
Throughout this section, we use the standard notation $x^y=y^{-1}xy$ and $[x,y]=x^{-1}y^{-1}xy$. We begin by constructing an auxiliary central extension of $\ZZ^2 \wr \ZZ$, which is reminiscent of the center-by-metabelian groups discussed in \cite{Hall} (see also \cite{Er}). 

Let 
\begin{equation}\label{Eq:A}
A=\{ a_i, b_i, c_i \mid i\in \ZZ\}
\end{equation}
and 
\begin{equation}\label{Eq:PresG}
G=\left\langle A, t \;\left|\; \begin{array}{c}
                              a_i^{t}=a_{i+1}, \; b_i^t=b_{i+1},\\
                              {[a_i, a_j]=[b_i,b_j]=1,\; [a_i,b_j]=c_{j-i},} \\ 
                              {[a_i, c_j]=[b_i,c_j]=[t, c_j]=c_j^2=1 }
                            \end{array}  \right.\right\rangle,
\end{equation}
where the relations are imposed for all $\forall\, i,j\in \ZZ$.

\begin{lemma}\label{Lem:1}
Let $G$ be the group defined above. 
\begin{enumerate}
\item[(a)] The center $Z(G)$ is a free group in the variety of abelian groups of exponent $2$, and $\{c_k \mid k\in \ZZ\}$ is a basis in $Z(G)$.
\item[(b)] $G/Z(G)\cong \ZZ^2\wr \ZZ$.
\end{enumerate}
\end{lemma}

\begin{proof}
Let $N$ be the free nilpotent group of rank $2$ with basis $\{ \tilde a_i, \tilde b_i\mid i\in \ZZ\}$. It is well-known that $[N,N]=Z(N)$ is a free abelian group freely generated by the basic commutators $[\tilde a_i, \tilde a_j]$, $[\tilde b_i, \tilde b_j]$ for $i<j$, $i,j\in \ZZ$, and $[\tilde a_i, \tilde b_j]$ for $i,j\in \ZZ$ (see \cite[Ch. 3, Sec. 1]{Neu} and references therein). Let $M\le [N,N]$ be the subgroup generated by the elements
$$
[\tilde a_i, \tilde a_j], \;\; [\tilde b_i, \tilde b_j],\;\; [\tilde a_i, \tilde b_j]^2,\;\; [\tilde a_i, \tilde b_{j}] [\tilde a_{i+\ell}, \tilde b_{j+\ell}]
$$
for all $i,j, \ell\in \ZZ$. It is straightforward to verify that $[N,N]/M$ is a free group in the variety of abelian groups of exponent $2$, freely generated by elements $c_k=[\tilde a_0, \tilde b_k]M$, $k\in \ZZ$. 

Since $M$ is central in $N$, we can consider the quotient group $N/M$. We denote the images of $\tilde a_i$ and $\tilde b_i$ ($i,j\in \ZZ$) in $N/M$ by $a_i$ and $b_j$, respectively. Our construction yields the presentation
$$
N/M = \left\langle A \left|\; [a_i, a_j]=[b_i,b_j]=1,\; [a_i,b_j]=c_{j-i},\; [a_i, c_j]=[b_i,c_j]=c_j^2=1 \right.\right\rangle,
$$
where $A$ is given by (\ref{Eq:A}) and the indices $i$, $j$ independently range in $\ZZ$. 

The map $\tilde a_i\mapsto \tilde a_{i+1}$, $\tilde b_i\mapsto \tilde b_{i+1}$ ($i\in \ZZ$) extends to an automorphism $N\to N$, which fixes $M$ setwise. Therefore, it induces an automorphism  $\phi\in \Aut(N/M)$ such that $\phi(a_i)=a_{i+1}$ and $\phi(b_i)=b_{i+1}$ for all $i\in \ZZ$. Note also that for every $i\in \ZZ$, we have
$$
\phi (c_i)= \phi ([\tilde a_0, \tilde b_i]M)= [\tilde a_1, \tilde b_{i+1}]M =[\tilde a_0, \tilde b_i]M = c_i.
$$ 
 Hence, the group $G$ given by (\ref{Eq:PresG}) is naturally isomorphic to the mapping torus $N/M\rtimes_\phi \ZZ$. In particular, $N/M$ naturally embeds in $G$. 

The subgroup $[N,N]/M=\langle \{ c_i \mid i\in \ZZ\}\rangle \le G$ is central in $G$ as all generators of $G$ commute with $c_i$ for all $i\in \ZZ$. The quotient group $G/([N,N]/M)$ has the presentation 
$$
\langle t, a_i, b_i \mid  [a_i, a_j]=[b_i,b_j]=[a_i,b_j]=1,\; a_i^{t}=a_{i+1}, \; b_i^t=b_{i+1}\rangle,
$$
which is one of the standard presentations of the wreath product $\ZZ^2 \wr \ZZ$. Since the latter group is centerless, we have $Z(G)=[N,N]/M$, and conditions (a) and (b) follow.
\end{proof}

For every subset $S\subseteq \ZZ$, we define a subgroup $Z_S\le Z(G)$ by the formula
$$
Z_S=\Big\langle \{c_i\mid i\in \ZZ\setminus S\}\cup \{c_jc_k, \mid, \; j,k\in S\}\Big\rangle 
$$ 
and let
$$
G_S=G/Z_S.
$$
For instance, we have $G_\emptyset \cong \ZZ^2\wr \ZZ$. 

\begin{lemma}\label{Lem:2}
For every non-empty subset $S\subseteq \ZZ$, the group $G_S$ is a central extension of the form 
$$
1\longrightarrow \ZZ_2 \longrightarrow G_S \longrightarrow \ZZ^2\wr \ZZ \longrightarrow 1,
$$
where the kernel of the extension is generated by $c_i$ for any $i\in S$.
\end{lemma}

\begin{proof}
Part (a) of Lemma \ref{Lem:1} and the definition of $Z_S$ imply that the image of $Z(G)$ in $G_S$ is isomorphic to $Z(G)/Z_S\cong \ZZ_2$. The result now follows from part (b) of Lemma \ref{Lem:1}. 
\end{proof}

We consider the \emph{shift action} of $\ZZ$ on $2^\ZZ$ by homeomorphisms defined by 
$$
n\circ S =\{ s-n\mid s\in S\} 
$$
for all $n\in \ZZ$ and all $S\in 2^\ZZ$. 

\begin{lemma}\label{Lem:3}
For every $S\subseteq \ZZ$ and every $n\in \ZZ$, the group $G_S$ is generated by the elements $a=a_0$, $b_n$, and $t$; in these generators, it has the presentation 
\begin{equation}\label{Eq:PresGS}
G_S=\left\langle a,b_n, t \;\left|\; \begin{array}{c}
                              {[a, a^{t^i}]=[b_n,b_n^{t^i}]=1\;\; \forall\, i\in \ZZ}\\ 
                              {[a,[a, b_n^{t^{i}}]]=[b_n,[a, b_n^{t^{i}}]]=[t,[a, b_n^{t^{i}}]]=[a, b_n^{t^{i}}]^2=1\;\; \forall \, i\in \ZZ}\\
                              {[a, b_n^{t^{j}}]= 1\;\; \forall\, j\in n\circ(\ZZ\setminus S),\;\; [a, b_n^{t^{k}}]=[a, b_n^{t^{\ell}}]\;\; \forall \, k,\ell\in n\circ S }
                            \end{array}  \right.\right\rangle.
\end{equation}
\end{lemma}

\begin{proof}
To obtain the required presentation of $G_S$, we proceed as follows. We first note that it suffices to impose the relations in the second and third rows of (\ref{Eq:PresG}) for $i=0$ only; all other relations are redundant and we omit them from the presentation of $G$. Further, we obtain a presentation of $G_S$ by adding the relations $c_i=1$ and $c_j=c_k$ for all $i \in \ZZ\setminus S$ and $j,k\in S$. We rename $a_0$ by $a$ and exclude all redundant generators using the relations $a_i=a^{t^i}$, $b_i=b_n^{t^{i-n}}$, and $c_i=[a, b_n^{t^{i-n}}]$. In particular, relations of the form $c_i=1$  and $c_j=c_k$ become $[a, b_n^{t^{i-n}}]=1$ and $[a, b_n^{t^{j-n}}]=[a, b_n^{t^{k-n}}]$, respectively, where $i \in \ZZ\setminus S$ and  $j,k\in S$. We replace the exponent $i-n$ with $i \in n\circ(\ZZ\setminus S)$ and $k-n$ with $k\in n\circ S$. Finally, removing the redundant relations occurring from the first row of relations in (\ref{Eq:PresG}), we obtain (\ref{Eq:PresGS}).
\end{proof}

\begin{cor}\label{Cor:iso}
For every $S\subseteq \ZZ$ and every $n\in \ZZ$, we have $G_S\cong G_{n\circ S}$. 
\end{cor}
\begin{proof}
By Lemma \ref{Lem:3}, the required isomorphism is given by the map $a\mapsto a$, $b_0\mapsto b_n$, $t\mapsto t$.
\end{proof}

\begin{lemma}\label{Lem:4}
The map $f\colon 2^\ZZ \to \G$ given by $f(S)=(G_S, \{a,b_0,t\})$ is injective and continuous.
\end{lemma}

\begin{proof}
To prove injectivity, let $S$ and $T$ be two distinct subsets of $\ZZ$. Without loss of generality, we can assume that there exists $i\in S\setminus T$. By Lemma \ref{Lem:3}, the relation $[a, b_0^{t^{i}}]=1$ holds in $G_T$. On the other hand, the element $[a, b_0^{t^{i}}]=c_i$ generates $Z(G_S)\cong \ZZ_2$ by Lemma~\ref{Lem:2}. In particular,  $[a, b_0^{t^{i}}]\ne 1$ in $G_S$. By the definition of marked groups, this implies that $(G_S, \{a,b_0,t\})$ and  $(G_T, \{a,b_0,t\})$ represent distinct elements of $\G$.

We now show that the map $f$ is continuous. To this end, we represent it as the composition of maps
$$
2^\ZZ\stackrel{\alpha}\to \mathcal N(G)\stackrel{\beta}\to \G,
$$
where $\alpha (S)=Z_S\lhd G$ for all $S\subseteq \ZZ$, and $\beta (N)=(G/N, \{a_0N, b_0N,tN\})$ for all $N\lhd G$.  Observe that $G$ is generated by $a_0$, $b_0$, and $t$; hence, the map $\beta$ is well-defined and continuous by Lemma~\ref{Lem:NG}. 

It remains to show that the map $\alpha $ is continuous. By the definition of the topologies on $2^\ZZ$ and $\mathcal N(G)$, this amounts to showing that for every finite subset $\mathcal F\subseteq  G$, there exists a finite subset $\mathcal E\subseteq \ZZ$ such that for every $S,T\in 2^\ZZ$ satisfying 
\begin{equation}\label{Eq:ST}
S\cap \mathcal E=T\cap \mathcal E, 
\end{equation}
we have $Z_S\cap \mathcal F=Z_T\cap \mathcal F$. We fix a finite subset $\mathcal F\subseteq  G$. By part (a) of Lemma \ref{Lem:1}, every element  $z\in \mathcal F\cap Z(G)$ can be uniquely decomposed as 
$$
z=\prod_{i\in \mathcal K(z)} c_i
$$
for some finite subset $\mathcal K(z)\subseteq \ZZ$. We let 
$$
\mathcal E = \bigcup_{z\in \mathcal F\cap Z(G)} \mathcal K(z).
$$

Suppose that $S$ and $T$ are subsets of $\ZZ$ satisfying (\ref{Eq:ST}). By the definition of $Z_S$, every element $z\in Z_S$ can be decomposed as 
\begin{equation}\label{Eq:zProd}
z=\left(\prod_{i\in \mathcal P} c_i \right)\cdot \left(\prod_{(j,k)\in \mathcal R} c_jc_k\right)
\end{equation}
for some finite subsets $\mathcal P\subseteq\ZZ\setminus S$ and $\mathcal R\subseteq S\times S$. We take such a decomposition with the minimal possible number of factors, which ensures that every $c_i$ appears in the decomposition at most once (for example, if the second product contains $c_jc_k$ and $c_jc_\ell$ for some $j, k, l\in S$, we can use the equality $c_jc_kc_jc_\ell=c_kc_\ell$ to reduce the number of factors). Since $\{c_i\mid i\in \ZZ\}$ is a basis in $Z(G)$, an element $c_i$ occurs in (\ref{Eq:zProd}) if and only if $i\in \mathcal K(z)\subseteq \mathcal E$. Combining this with (\ref{Eq:ST}), we obtain that $\mathcal P\subseteq\ZZ\setminus T$ and $\mathcal R\subseteq T\times T$. Consequently, $z\in Z_T$. This proves the inclusion $Z_S\cap \mathcal F \le Z_T\cap \mathcal F$. The proof of the opposite inclusion is symmetric, so we obtain $Z_S\cap \mathcal F = Z_T\cap \mathcal F$, as required. This concludes the proof of the continuity of $\alpha$ and the lemma. 
\end{proof}

Recall that the action of a group $H$ by homeomorphisms on a topological space $X$ is said to be \emph{topologically transitive} if, for any non-empty open subsets $U, V\subseteq X$, there exists $g\in G$ such that $gU\cap V\ne \emptyset$. If $X$ is Polish, this condition is equivalent to the existence of a dense orbit.  

\begin{proof}[Proof of Theorem \ref{main}]
It is well-known and straightforward to prove that the shift action of $\ZZ$ on $2^\ZZ$ is topologically transitive. Thus, there exists $S\in 2^\ZZ$ with a dense orbit. In particular, the orbit of $S$ has no isolated points. By Lemma \ref{Lem:4} and Corollary \ref{Cor:iso}, the image of the orbit of $S$ under the map $f$ has no isolated points and belongs to the isomorphism class $[G_S]\subseteq \G$. It follows that $[G_S]$ is non-discrete. Since the isomorphism class of every finitely generated group in $\G$ is either discrete or has no isolated points (see \cite[Corollary 6.1]{Osi21a}), we obtain that $G_S$ is condensed. 
\end{proof}

\subsection{Elementarily equivalent center-by-metabelian groups.}

We are now ready to prove the corollaries announced in the introduction. 

\begin{proof}[Proof of Corollary \ref{Cor:non-smooth}] 
This is an immediate consequence of Theorem \ref{main} and \cite[Proposition 2.7]{Osi21a} applied to the subset of $\G$ consisting of all finitely generated, center-by-metabelian marked groups. Indeed, the aforementioned proposition states that an isomorphism-invariant closed subset $S$ of $\G$ contains a condensed group if and only if the isomorphism relation on $S$ is not smooth.
\end{proof}

\begin{proof}[Proof of Corollary \ref{Cor:Non-Geom}]
    As observed in \cite[Proposition 6.2]{Osi21a}, finitely generated abelian-by-nilpotent (in particular, finitely generated metabelian) groups cannot be condensed. Thus, $\ZZ^2 \wr \ZZ$ is not condensed. Combining this with Theorem \ref{main}, we obtain the desired result. 
\end{proof}

The proof of Corollary \ref{Cor:EE} follows the same idea as that of Proposition 2.3 in \cite{Osi21a} and involves proving a zero-one law for the set $f(2^\ZZ)=\{ (G_S, \{a,b_0, t\})\mid S\subseteq \ZZ\}$ analogous to Theorem 2.2 in \cite{Osi21a}. Unfortunately, we cannot apply the latter theorem in our settings as $f(2^\ZZ)$ is not isomorphism-invariant in the terminology of \cite{Osi21a}. To address this problem, we first prove the proposition below, which seems to be of independent interest.

\begin{definition}
Let $H$ be a group acting on the space $\G$. We say that the action is \emph{isomorphism-preserving} if, for every marked group $(K,X)\in \G$, the $H$-orbit of $(K,X)$ belongs to the isomorphism class $[K]$.
\end{definition}

Recall that a subset of a topological space is \emph{perfect} if it is closed and has no isolated points.

\begin{prop}\label{Prop:EE}
Let $\mathcal C$ be a perfect subspace of $\G$. Suppose that $Homeo(\mathcal C)$ contains an isomorphism-preserving, topologically transitive subgroup. Then $\mathcal C$ contains a subset $\mathcal C_0$ of cardinality $|\mathcal C_0|=2^{\aleph_0}$ such that, for any $(K,X), (L,Y)\in \mathcal C_0$, we have $K\equiv L$.
\end{prop}
\begin{proof}
The proof of the proposition repeats the main steps of the proofs of Theorem 2.2 and Proposition 2.3 in \cite{Osi21a}, so we keep it brief. 

For every first-order sentence $\sigma $ in the language of groups $\mathcal L$, let 
$$
\ModS(\sigma)=\{ (K,X)\in \mathcal C\mid K\models \sigma\}.
$$
By \cite[Proposition 5.1]{Osi21a}, $\ModS(\sigma)$ is always a Borel subset of $\mathcal C$. Since $\mathcal C$ is a closed subspace of the Polish space $\G$, it is a Polish space itself. Let $H\le Homeo(\mathcal C)$ be an isomorphism-preserving, topologically transitive subgroup. By the zero-one law for topologically transitive group actions on Polish spaces (see \cite[Theorem 8.46]{Kec}), every $H$-invariant Borel subset of $\mathcal C$ is either meager or comeager. In particular, $\ModS(\sigma )$ is either meager or comeager for every sentence $\sigma $ in $\mathcal L$. 

Let $Th^{gen}(\mathcal C)$ denote the set of all sentences $\sigma $ in $\mathcal L$ such that $\ModS(\sigma )$ is comeager in $\mathcal C$. Since the intersection of a countable collection of comeager sets is comeager, there is a comeager set $\mathcal C_0\subseteq \mathcal C$ such that, for every $(K,X)\in \mathcal C_0$, the group $K$ is a model of $Th^{gen}(\mathcal C)$. On the other hand, for every sentence $\sigma $, we have $\ModS(\sigma)\cup \ModS(\lnot\sigma)=\mathcal C$. By the Baire category theorem, $\mathcal C$ cannot be represented as a union of two meager sets. Hence, either $\ModS(\sigma)$ or $\ModS(\lnot\sigma)$ must be comeager. This implies that the theory $Th^{gen}(\mathcal C)$ is complete. Since all models of a complete theory are elementarily equivalent, we have $K\equiv L$ for all $(K,X), (L,Y)\in \mathcal C_0$. 

Finally, recall that every comeager subset of a Polish space contains a dense $G_\delta$-subset by the Baire category theorem. In particular, $\mathcal C_0$ contains a dense $G_\delta$-subset $\mathcal C_1$. Since $\mathcal C_0$ is perfect, $\mathcal C_1$ has no isolated points. Every $G_\delta $-subspace of a Polish space is Polish (see \cite[Theorem~I.3.11]{Kec}), hence $\mathcal C_1$ is Polish. Being a Polish space without isolated points, $\mathcal C_1$ has the cardinality of the continuum by \cite[Corollary~I.6.3]{Kec}. Therefore, we have $|\mathcal C_0|=2^{\aleph_0}$.  
\end{proof}

\begin{proof}[Proof of Corollary \ref{Cor:EE}]
Let $\mathcal C=f(2^\ZZ)$, where $f$ is the map from Lemma \ref{Lem:4}. Note that $\mathcal C$ consists of groups of the form (\ref{ext}) by Lemma \ref{Lem:2}. Since $2^\ZZ$ is compact and $f$ is continuous, $\mathcal C$ is compact; in particular, $\mathcal C$ is a closed subset of $\G$. Note also that $\mathcal C$ has no isolated points being the image of a space without isolated points under a continuous and injective map.  

The shift action of $\ZZ$ on $2^\ZZ$ induces an action of $\ZZ$ on $\mathcal C$ by the rule
$$
n (G_S, \{a, b_0, t\}) = (G_{n\circ S}, \{a, b_0, t\}).
$$
Since the map $f\colon 2^\ZZ\to \mathcal C$ is bijective and continuous with compact domain, it is a homeomorphism; this implies that the action of $\ZZ$ on $\mathcal C$ defined above is continuous; in addition it is topologically transitive as so is the action of $\ZZ$ on $2^\ZZ$. Further, the action of $\ZZ$ on $\mathcal C$ is isomorphism-preserving by Corollary \ref{Cor:iso}. 

Applying Proposition \ref{Prop:EE}, we obtain a subset $\mathcal C_0$ of cardinality $|\mathcal C_0|=2^{\aleph_0}$ such that, for any $(G_S, \{a, b_0, t\}), (G_T, \{a, b_0, t\})\in \mathcal C_0$, we have $G_S\equiv G_T$. Clearly, for every finitely generated group $K$, the isomorphism class $[K]$ in $\G$ is countable. Therefore, there are $2^{\aleph_0}$ pairwise non-isomorphic, elementarily equivalent groups among $G_S$, $S\subseteq \ZZ$. 
\end{proof}


\begin{thebibliography}{10}

\bibitem{AhZ} 
G.\,Ahlbrandt, M.\,Ziegler, {Quasi finitely axiomatizable totally categorical theories}, Ann. Pure Appl. Logic, {\bf 30:1} (1986), pp.\,63--82.

\bibitem{AKNS} 
M. Aschenbrenner, A. Khelif,  E. Naziazeno, T. Scanlon, The logical complexity of finitely generated commutative rings. Int. Math. Res. Not. IMRN 2020, no. 1, 112-166.

\bibitem{ALM1} 
N. Avni, A. Lubotzky, C. Meiri, First order rigidity of non-uniform higher rank arithmetic lattices, Invent. Math. 217 (2019), 219--240.

\bibitem{AvM}  
N. Avni, C. Meiri, On the model theory of higher rank arithmetic groups, Duke Math. J. 172 (2023), no. 13, 2537--2590.

\bibitem{Baumslag} 
G. Baumslag, F. B. Canonito, D. Robinson, The algorithmic theory of finitely generated metabelian groups, Trans. Amer. Math. Soc. 344 (1994), no. 2, 629--648.

\bibitem{BCM} 
W. Baur, G. Cherlin, A. Macintyre, Totally categorical groups and rings, J. Algebra 57 (1979), no. 2, 407--440.

\bibitem{BG} 
E. Bunina, P. Gvozdevsky, Regular bi-interpretability and finite axiomatizability of Chevalley groups, arXiv:2311.01954v5 [math.GR].

\bibitem{DM1} 
E. Daniyarova, A. Myasnikov, Theory of Interpretations I. Foundations, arXiv:2511.13810 [math.LO].

\bibitem{DM2} 
E. Daniyarova, A. Myasnikov, Groups elementarily equivalent to metabelian Baumslag-Solitar groups and regular bi-interpretability, Ann. Pure Appl. Logic 177 (2026), no. 5, 103695.

\bibitem{DP} 
P. Dittmann, F. Pop, Characterizing finitely generated fields by a single field axiom, Ann. of Math. 198 (2023), no. 3, 1203--1227.

\bibitem{Er} 
A. Erschler, Not residually finite groups of intermediate growth, commensurability and non-geometricity, J. Algebra 272 (2004), 154--172.

\bibitem{GPS} 
F. Grunewald, P. Pickel, D. Segal, Polycyclic groups with isomorphic finite quotients, Ann. of Math. 111 (1980), no. 1, 155--195.

\bibitem{Gri}
R. Grigorchuk, Degrees of growth of finitely generated groups and the theory of invariant means, Izv. Akad. Nauk SSSR Ser. Mat. 48 (1984), no. 5, 939--985.

\bibitem{GMO} 
A. Garreta, A. Myasnikov, D. Ovchinnikov, Full rank presentations and nilpotent groups: structure, Diophantine problem, and genericity, J. Algebra 556 (2020), 1--34.

\bibitem{Hall}
P. Hall, Finiteness conditions for soluble groups, Proc. London Math. Soc. 4 (1954), no. 3, 419--436.

\bibitem{Hodges} 
W. Hodges, { Model Theory}, Cambridge University Press, 1993.

\bibitem{Kaye} 
R. Kaye, Models of Peano Arithmetic, Oxford Logic Guides 15, Oxford University Press, 1991.

\bibitem{Kec}
A. Kechris, {  Classical descriptive set theory}, Graduate Texts in Mathematics 156, Springer-Verlag, New York, 1995.

\bibitem{Khelif}  
A. Khelif, Bi-interpretabilit\'e et structures QFA: \'etude des groupes r\'esolubles et des anneaux commutatifs, C. R. Acad. Sci. Paris Ser. I 345 (2007), 59--61.

\bibitem{KM1}  
O. Kharlampovich, A. Miasnikov, Groups elementarily equivalent to a finitely generated free metabelian group, Groups Geom. Dyn. 19 (2025), no. 2, 681--710.

\bibitem{KM3} 
O. Kharlampovich, A. Myasnikov, Elementary theory of free non-abelian groups, J. Algebra 302 (2006), no. 2, 451--552.

\bibitem{KMS} 
O. Kharlampovich, A. Myasnikov, M. Sohrabi, Rich groups, weak second order logic and applications, in: Groups and Model Theory, GAGTA Book 2, De Gruyter, 2021, 127--193.

\bibitem{Lasserre} 
C. Lasserre, Polycyclic-by-finite groups and first-order sentences, J. Algebra 396 (2013), 18--38.

\bibitem{Lasserre2} 
C. Lasserre, R. J. Thomson's groups F and T are bi-interpretable with the ring of integers, J. Symbolic Logic 79 (2014), no. 3, 693--711.

\bibitem{Malcev} 
A. Malcev, Constructive algebras I, Uspekhi Mat. Nauk 16 (1961), no. 3, 3--60.

\bibitem{Mat} 
Y. Matiyasevich, Enumerable sets are Diophantine, Dokl. Akad. Nauk SSSR 191 (1970), 279--282.

\bibitem{MN} 
A. Miasnikov, A. Nikolaev, Nonstandard polynomials: algebraic properties and elementary equivalence, arXiv:2409.14467.

\bibitem{MOW}
A. Minasyan, D. Osin, S. Witzel, Quasi-isometric diversity of finitely generated groups, J. Topology 14 (2021), no. 2, 488--503.

\bibitem{Neu}
H. Neumann, Varieties of groups, Springer-Verlag, New York, 1967.

\bibitem{Nies1}  
A.\,Nies, {  Comparing quasi-finitely axiomatizable and prime groups}, J. Group Theory 10 (2007), 347--361.

\bibitem{Nies2} 
A.\,Nies, {  Describing groups}, Bull. Symbolic Logic 13 (2007), no. 3, 305--339.

\bibitem{O2} 
F. Oger, Elementary equivalence of a polycyclic-by-finite group and its profinite completion, Arch. Math. 52 (1989), 521--525.

\bibitem{OgerSabbagh} 
F. Oger, G. Sabbagh, Quasi-finitely axiomatizable nilpotent groups, J. Group Theory 9 (2006), 95--106.

\bibitem{Osi21a}
D. Osin, A topological zero-one law and elementary equivalence of finitely generated groups, Ann. Pure Appl. Logic 172 (2021), no. 3, 102915.

\bibitem{Osi21b}
D. Osin, Condensed groups in product varieties, J. Group Theory 24 (2021), no. 4, 753--763.

\bibitem{Pickel1} 
P. Pickel, Finitely generated nilpotent groups with isomorphic finite quotients, Trans. Amer. Math. Soc. 160 (1971), 327--341.

\bibitem{Pickel2}  
P. Pickel, Metabelian groups with the same finite quotients, Bull. Austral. Math. Soc. 11 (1974), 115--120.

\bibitem{Prest} 
M. Prest, {  Model Theory and Modules}, London Mathematical Society Lecture Note Series 130, Cambridge University Press, 1988.

\bibitem{Rabin} 
M. Rabin, Computable algebra: general theory and theory of computable fields, Trans. Amer. Math. Soc. 95 (1960), no. 2, 341--360.

\bibitem{Robinson} 
R.\,Robinson, {  Undecidable rings}, Trans. Amer. Math. Soc. 70 (1951), 137--159.

\bibitem{Rogers} 
H. Rogers, {  The Theory of Recursive Functions and Effective Computability}, McGraw-Hill, 1967.

\bibitem{SW}  
G. Sabbagh, J. S. Wilson, Polycyclic groups, finite images, and elementary equivalence, Arch. Math. 57 (1991), 221--227.

\bibitem{Segal} 
D. Segal, {  Words: Notes on Verbal Width in Groups}, LMS Lecture Note Series 361, Cambridge University Press, 2009.

\bibitem{Sela} 
Z. Sela, Diophantine geometry over groups VI: The elementary theory of a free group, Geom. Funct. Anal. 16 (2006), no. 3, 707--730.

\bibitem{Soare} 
R. Soare, {  Recursively enumerable sets and degrees}, Springer-Verlag, Berlin, 1987.

\bibitem{W}
J. Williams, Isomorphism of finitely generated solvable groups is weakly universal, J. Pure Appl. Algebra 219 (2015), no. 5, 1639--1644.

\end{thebibliography}
\end{document}